\documentclass[a4paper, 11pt]{article}
\usepackage{amsthm,amsfonts,amsmath,amssymb,amscd}
\usepackage{color,graphicx,ifthen,latexsym,pstricks,stmaryrd,subfigure}
\usepackage[all]{xy} 
\usepackage[T1]{fontenc}

\def\Hom{\mathop{\rm Hom}\nolimits}

\def\ker{\mathop{\rm ker}\nolimits}
\def\im{\mathop{\rm im}\nolimits}
\def\coker{\mathop{\rm coker}\nolimits}

\def\symp{\mbox{\sl Symp}}

\def\note#1{\marginpar{\raggedright\if@twoside\ifodd\c@page\raggedleft\fi\fi\sf\scriptsize RMK: #1}}
\DeclareMathOperator{\pr}{pr}

\DeclareMathOperator{\Ext}{Ext}
\DeclareMathOperator{\Gl}{Gl}
\DeclareMathOperator{\Cl}{Cl}

\DeclareMathOperator{\id}{id}

\DeclareMathOperator{\Spec}{Spec}
\DeclareMathOperator{\Cone}{Cone}

\newcommand{\stcky}[1]{\text{{\boldmath{$#1$}}}}

\newcommand{\mf}{\mathfrak}
\newcommand{\mr}{\mathrm}

\newcommand{\mc}{\mathcal}

\newcommand{\N}{\mathbb{N}}
\newcommand{\Q}{\mathbb{Q}}
\newcommand{\R}{\mathbb{R}}
\newcommand{\Z}{\mathbb{Z}}
\newcommand{\C}{\mathbb{C}}

\renewcommand{\iff}{if and only if }
\newcommand{\wrt}{with respect to }
\newcommand{\st}{such that }

\newcommand{\Id}{\mr{Id}}

\newcommand{\be}{\begin{equation}}
\newcommand{\ben}{\begin{equation}\nonumber}
\newcommand{\ee}{\end{equation}}

\newcommand{\OLD}[1]{}
\newcommand{\TT}{\mathbb{T}}
\newcommand{\ZZ}{\mathbb{Z}}

\newcommand{\RR}{\mathbb{R}}
\newcommand{\PP}{\mathbb{P}}
\newcommand{\CC}{\mathbb{C}}

\newcommand{\XX}{\mathcal{X}}

\newcommand{\skalp}[1]{\langle #1 \rangle}

\newcommand{\veraltet}[1]{}

\newcommand{\Stacky}{\mathbf{\Sigma}}

\newcommand{\quotient}[2]{#1 / #2}

\newcommand{\Bhalfarray}{
\setlength{\arraycolsep}{0.5\arraycolsep}
\begin{array}
}
\newcommand{\Ehalfarray}{
\end{array}
\setlength{\arraycolsep}{2\arraycolsep}
}
\newcommand{\Btightarray}{
\setlength{\savearraycolsep}{\arraycolsep}
\arraycolsep1pt
\begin{array}
}
\newcommand{\Etightarray}{
\end{array}
\setlength{\arraycolsep}{\savearraycolsep}
}
\newcommand{\Epicture}{\end{pspicture}}
\newcommand{\ownpsunit}{\psset{unit=0.05\textwidth}}
\newcommand{\qcircle}{\pscircle[fillstyle=solid,fillcolor=white]}

\newcommand{\Bpicture}{\centering\ownpsunit\begin{pspicture}}
\newcommand{\refmoritadiag}{(M1){}}
\newcommand{\refmoritadiagp}{(M1'){}}
\newcommand{\refmoritasurj}{(M2){}}
\newcommand{\figproj}{
\Bpicture(8,4)
\psline(1,1)(3,1)(1,3)(1,1)
\psline{->}(1.6,1)(1.6,1.4)
\psline{->}(1,1.6)(1.4,1.6)
\psline{->}(2,2)(1.7,1.7)
\uput[180](1,2){$F_1$}
\uput[-90](2,1){$F_2$}
\uput[45](2,2){$F_3$}
\psset{origin={4,0}}
\multips(0,0)(0,1){5}{
\multips(0,0)(1,0){5}{\pscircle(0,0){0.02}}}
\psline(0,0)(2,2)(2,4)(2,2)(4,2)
\qcircle(1,1){0.1}
\qcircle(2,3){0.1}
\qcircle(3,2){0.1}
\uput[-90](7,2){$u_1$}
\uput[180](6,3){$u_2$}
\uput[-45](5,1){$u_3$}
\Epicture
}

\newcommand{\figweightproj}{
\Bpicture(8.5,4)
\psline(0.5,1.5)(3.5,1.5)(0.5,3)(0.5,1.5)
\psline{->}(0.5,2.25)(0.9,2.25)
\psline{->}(2,1.5)(2,1.9)
\psline{->}(1.7,2.4)(1.5,2)
\uput[180](0.5,2.25){$F_1$}
\uput[-90](2,1.5){$F_2$}
\uput[60](1.7,2.4){$F_3$}
\psset{origin={4.5,0}}
\multips(0,0)(0,1){5}{ \multips(0,0)(1,0){5}{
\pscircle(0,0){0.02} }}
\psline(0.75,-0.5)(2,2)(2,4)(2,2)(4,2)
\qcircle(1,0){0.1}
\qcircle(2,3){0.1}
\qcircle(3,2){0.1}
\uput[-90](7.5,2){$u_1$}
\uput[180](6.5,3){$u_2$}
\uput[-30](5.5,0){$u_3$}
\Epicture
}

\newcommand{\narrenkappe}{
\Bpicture(12,4)
\qcircle(2,2){1.5}
\psline(2,2)(2,0.5)
\qdisk(2,2){0.1}
\pscustom{
\psarc(6,2){1.5}{90}{280}
\psline(6,2)(6,2.01)}
\pscustom{
\psline(6,2)(6,2.01)
\psarc(6,2.1){1.4}{260}{90}}
\psarc[fillstyle=solid,arcsepA=0.5](6,2.1){1.4}{260}{90}
\qdisk(6,2){0.1}
\psarcn{->}(6,2){1}{270}{200}
\psset{origin={10,3.5}}
\SpecialCoor
\psline(3.5;250)(0,0)(3.5;290)
\psline(0,0)(3.5;260)
\psarc(0,0){3.5}{250}{290}
\qdisk(0,0){0.1}
\Epicture
}

\newenvironment{rem}{\begin{trivlist}\item[]{\bf Remark:}\setlength{\parindent}{0pt}}{\end{trivlist}}
\newenvironment{ex}{\begin{trivlist}\item[]{\bf Example:}\setlength{\parindent}{0pt}}{\end{trivlist}}

\newenvironment{thm*}{\begin{trivlist}\item[]{\bf Theorem.}}{\end{trivlist}}
\newenvironment{dfn*}{\begin{trivlist}\item[]{\bf Definition.}}{\end{trivlist}}
\newenvironment{prp*}{\begin{trivlist}\item[]{\bf Proposition.}}{\end{trivlist}}
\newenvironment{lem*}{\begin{trivlist}\item[]{\bf Lemma.}}{\end{trivlist}}
\newenvironment{cor*}{\begin{trivlist}\item[]{\bf Corollary.}}{\end{trivlist}}

\newlength{\fracw}
\newlength{\savearraycolsep}

\newtheorem{thm}{Theorem}[section]
\newtheorem{dfn}[thm]{Definition}
\newtheorem{prp}[thm]{Proposition}
\newtheorem{lem}[thm]{Lemma}

\begin{document}
\title{On complex and symplectic toric stacks}
\author{Andreas Hochenegger and Frederik Witt}
\date{}
\maketitle

\begin{quote}
Toric varieties play an important r\^ole both in symplectic and complex geometry. In symplectic geometry, the construction of a symplectic toric manifold from a smooth polytope is due to Delzant~\cite{del}. In algebraic geometry, there is a more general construction using fans rather than polytopes. However, in case the fan is induced by a smooth polytope Audin~\cite{aud} showed both constructions to give isomorphic projective varieties. For rational but not necessarily smooth polytopes the Delzant construction was refined by Lerman and Tolman~\cite{lt}, leading to symplectic toric orbifolds or more generally, symplectic toric DM stacks~\cite{lm}. We show that the stacks resulting from the Lerman--Tolman construction are isomorphic to the stacks obtained by Borisov et al.~\cite{bcs} in case the stacky fan is induced by a polytope. No originality is claimed (cf.\ also the article by Sakai~\cite{s}). Rather we hope that this text serves as an example driven introduction to symplectic toric geometry for the algebraically minded reader.
\end{quote}
%
%
%
%
%
\section{Delzant's theorem}\label{delzant}
We briefly describe the symplectic construction of a toric variety starting from a rational polytope. Good references are~\cite{aud} and~\cite{gui}. In this section we assume manifolds, tensors, maps between manifolds etc.\ to be differentiable (i.e.\ of class $C^\infty$) unless mentioned otherwise.
\paragraph{Symplectic toric manifolds.}
Let $U$ be a manifold. A {\em symplectic form} for $U$ is a closed, non--degenerate $2$--form $\omega$, that is $d\omega=0$, and the natural map sending a vector field $v\in\mf{X}(U)$ to the $1$--form $v\llcorner\omega=\omega(v,\cdot)\in\Omega^1(U)$ is a linear isomorphism. In particular, $U$ must be even dimensional. We call the pair $(U,\omega)$ a {\em symplectic manifold}. The automorphism group of a symplectic manifold, the group of {\em symplectomorphisms} $\symp(U,\omega)$, consists of diffeomorphisms preserving the symplectic form under pullback. Symplectomorphisms exist in abundance. Indeed, take any smooth function $H\in C^\infty(U)$ and define the associated {\em Hamiltonian vector field} $v_H$ by $\omega(v_H,\cdot)=-dH$. Then for $U$ compact the flow of $v_H$ gives a curve in $\symp(U,\omega)$.

\begin{dfn}\label{hamac}
Let $G$ be a Lie group. An action of $G$ on a symplectic manifold $(U,\omega)$ is {\bf hamiltonian} if
\begin{itemize}
	\item there is a $G$--equivariant map $\mu:U\to\mf{g}^\ast$ from the manifold to the dual of the Lie algebra $\mf{g}$ of $G$ ($G$ acting via the coadjoint representation). We call $\mu$ the {\bf moment map} of the action.
	\item the fundamental vector fields $v^\sharp$ induced by $v\in\mf{g}$ satisfy $v^\sharp\llcorner\omega=-d\langle\mu,v\rangle$ (where $\langle\cdot\,,\cdot\rangle$ denotes evaluation of $\mu\in C^\infty(U,\mf{g}^*)$ on $v\in\mf{g}$).
\end{itemize}
\end{dfn}

\begin{ex}
Consider $\C^d$ with its standard symplectic form $\omega_0=i\sum dz_k\wedge d\bar z_k/2$. Let $T^d=S^1\times\ldots\times S^1=(\R/\Z)^d$ denote the compact (as opposed to algebraic) torus of dimension $d$. Then $t=(t_1,\ldots,t_d)\in T^d$ acts on $z\in\C^d$ via $t.z=(t_1z_1,\ldots,t_dz_d)$. For the standard basis $e_1,\ldots,e_d$ of the Lie algebra $\mf{t}^d\cong\R^d$, the induced fundamental vector fields are $e_k^\sharp(z)=i(z_k\partial_{z_k}-\bar{z}_k\partial_{\bar{z}_k})$, hence $e_k^\sharp\llcorner\omega=-(z_kd\bar{z}_k+\bar{z}_kdz_k)/2=-d|z_k|^2$. Therefore, the action is hamiltonian with moment map
\ben
\mu_0(z)=\tfrac{1}{2}(|z_1|^2,\ldots,|z_d|^2).
\ee
More generally, if $G\subset T^d$ acts on $\C^d$ as a subgroup of $T^d$, then the action is hamiltonian with moment map $\mu_G=\iota^\ast\circ\mu$, where $\iota^\ast$ is the dual of the natural inclusion of Lie algebras $\iota:\mf{g}\hookrightarrow\mf{t}^d$.
\end{ex}

Hamiltonian actions by compact tori are particularly interesting because of the following

\begin{thm}{\bf (Atiyah~\cite{at}/Guillemin--Sternberg~\cite{gs})}
For the hamiltonian action of a compact torus with moment map $\mu$ on a compact, connected symplectic manifold, the set of fixed points of the action is a finite union of submanifolds $C_1,\ldots,C_r$. On each of these submanifolds, $\mu(C_j)\equiv\eta_j$ is constant and the image of $\mu$ is the convex hull of the points $\eta_j$. 
\end{thm}

Since the convex hull of a finite set of points in a real vector space is a polytope, one refers to the image of the moment map as the {\em moment polytope}. 

\begin{ex}
In continuation of the previous example, consider the complex projective space $\PP^d\cong S^{2d+1}/S^1$. The $T^{d+1}$--action on $\C^{d+1}$ preserves the unit sphere $S^{2d+1}$ on which $t\in S^1$ acts via $(t,\ldots,t)\in T^{d+1}$. The standard symplectic form on $\C^{d+1}$ descends to the quotient $S^{2d+1}/S^1$ and induces the well--known Fubini--Study form $\omega_{FS}$. We get a hamiltonian $T^d$--action for $(\PP^d,\omega_{FS})$ by sending $(t_1,\ldots,t_d)\in T^d$ to $(1,t_1,\ldots,t_d)\in T^{d+1}$ and using the $T^{d+1}$--action on the sphere. Indeed, if $s:\mf{t}^d\to\mf{t}^{d+1}$ denotes the resulting inclusion at Lie algebra level, then $\mu_{T^d}\circ\pi=s^\ast\circ\mu_0|_{S^{2d+1}}$ (with $\pi:S^{2d+1}\to\PP^d$ the natural projection), that is,
\ben
\mu_{T^d}([z_0:\ldots:z_d])=\tfrac{1}{2\sum_{k=0}^d|z_k|^2}(|z_1|^2,\ldots,|z_d|^2)
\ee
is a moment map for this action. In particular, the moment polytope is the simplex given by the images under $\mu_{T^d}$ of the fixed points $[1:0:\ldots:0],\ldots,[0:0:\ldots:1]$. 
\end{ex}

\begin{dfn}
A {\bf symplectic toric manifold} is a symplectic manifold $(U,\omega)$ of dimension $2d$ together with an effective hamiltonian action by a compact $d$--torus.
\end{dfn}

This is the symplectic counterpart of a complex toric variety (defined at the beginning of Section~\ref{anastack}). Note that for an effective action of $T^l$ we need $l\leq\dim U/2$ (cf.\ for instance Theorem 1.3 in~\cite{gui}) so that $d$ is the maximal dimension. Furthermore, since in this case, the moment map must be a submersion at some point (i.e.\ the differential is surjective at that point), the moment polytope is $d$--dimensional.

\medskip

For a compact torus $T^d$ we denote by $N\cong\Z^d$ the natural lattice inside the Lie algebra $\mf{t}\cong\R^d$. The dual lattice $\Hom(N,\Z)$ is written $M$. Further, let $\Delta\subset\mf{t}^\ast$ be a polytope with $m$ facets (i.e.\ codimension $1$ faces) and open interior (the vertices are {\em not} necessarily lattice points). 

\begin{dfn}
The polytope $\Delta$ will be called {\bf rational} if it can be written 
\be\label{etadata}
\Delta=\bigcap_{j=1}^m\{\alpha\in\mf{t}^\ast\,|\,\langle\alpha,u_j\rangle\geq-\eta_j\in\R\}
\ee
for $u_j\in N$, $j=1,\ldots,m$. In this case, we take the $u_j\in N$ to be primitive and inward pointing. Furthermore, we say that $\Delta$ is {\bf smooth} if for any vertex $w\in\Delta$, the subset of vectors $u_{j_1},\ldots,u_{j_d}$ corresponding to facets meeting at $w$, forms a basis of $N$.
\end{dfn}

For example, the moment polytope of $\PP^d$, or more generally of any other toric symplectic manifold, is smooth. Conversely:

\begin{thm}{\bf(Delzant~\cite{del})}
Any smooth polytope $\Delta$ arises as the moment polytope of a symplectic toric manifold $U_\Delta$. Furthermore, two symplectic toric manifolds are equivariantly symplectomorphic \iff their associated moment polytopes can be mapped to each other by translation.
\end{thm}

\begin{rem}
For a given polytope $\Delta$ it follows from Delzant's construction that $U_\Delta$ admits a natural compatible complex structure and is therefore K\"ahler. In fact, $U_\Delta$ is biholomorphic to any complex projective toric variety associated with a polytope with vertices in $N$ and inducing the same normal fan as $\Delta$. Furthermore, the euclidean volume of $U_\Delta$ is proportional to the euclidean volume of $\Delta$ (cf.\ for instance Theorem 2.10 in~\cite{gui}).
\end{rem}
\paragraph{The Lerman--Tolman theorem.}
From the view point of toric geometry it is natural to extend Delzant's theorem to the case of rational polytopes. Namely, any rational polytope (with vertices in $N$) gives a complex projective toric variety which is an {\em orbifold}, i.e.\ has at worst quotient singularities. Any effective orbifold is of the form $U/G$, where $U$ is a manifold and $G$ a compact connected Lie group acting  effectively and locally freely on $U$, that is, with finite isotropy groups (cf.\ Corollary 2.16 and Theorem 2.19 in~\cite{mm}; for the noneffective case, see~\cite{hm}).

\medskip

On an orbifold $X$, differential forms, vector fields etc.\ can be defined either by using the isomorphism with $U/G$ or in terms of an {\em orbifold atlas}. This is a collection $(U_\alpha,G_\alpha,f_\alpha:U_\alpha\to X)$ \st $G_\alpha$ is a discrete group acting effectively on the manifold $U_\alpha$ and $f_\alpha$ descends to a homeomorphism $U_\alpha/G_\alpha\to X$ onto an open subset $V_\alpha\subset X$. These data are required to satisfy the following conditions:
\begin{itemize}
	\item The collection $\{V_\alpha\}$ is an open covering of $X$.
	\item If $x_\alpha\in U_\alpha$ and $x_\beta\in U_\beta$ get mapped to the same point in $X$, i.e.\ $f_\alpha(x_\alpha)=f_\beta(x_\beta)$, then there exists a germ of a diffeomorphism $f_{\alpha\beta}$ from some connected open neighbourhood of $x_\alpha$ to an open neighbourhood of $x_\beta$ \st $f_\beta\circ f_{\beta\alpha}=f_\alpha$.
\end{itemize}
A differential form of degree $k$ on an orbifold is then given by a collection of $\lambda_\alpha\in\Omega^k(U_\alpha)$ which agree on overlaps and which are invariant under the induced action of $G_\alpha$. Similarly, one can define vector fields. Hamiltonian group actions are more delicate to define (cf.~\cite{hs}), but nevertheless there is a natural notion of a {\em symplectic toric orbifold} (see also Definition~\ref{symtoricstack} of a symplectic toric stack).

\medskip

Now with each point $y\in V_\alpha\subset X$ we can associate the isotropy group of the $G_\alpha$--orbit $f_\alpha^{-1}(y)$, which is well--defined up to conjugation. In the case of a symplectic toric orbifold for instance, there exists an integer $n_F$ for any each open facet $F$ (i.e.\ the relative interior of the facet $\bar F$) in $\Delta$ such that the isotropy group of any $y$ in the preimage of $F$ under the moment map is $\Z/n_F\Z$. Attaching this integer to the open facet as an additional datum associates a {\em labelled} polytope $\underline\Delta$ with any symplectic toric orbifold. Two such labelled polytopes are {\em isomorphic} if they differ only by a translation \st the corresponding facets carry the same labels.

\begin{thm}{\bf(Lerman--Tolman~\cite{lt})}\label{lertol}
There is a 1--1 correspondence between isomorphism classes of labelled rational polytopes and symplectic toric orbifolds up to equivariant symplectomorphism. 
\end{thm}

\begin{ex}
Consider the polytope $\underline\Delta\subset\R^\ast$ given by the interval $[0,1]$ with non--trivial labels $n_1=k$ and $n_2=1$ at the facets $0$ and $1$, whence $\eta_1=0$ and $\eta_2=1$ in~\eqref{etadata}. Let $e_{1,2}$ be the standard basis of $\R^2$ and define $\beta:\R^2\to\R$ by $\beta(e_j)=n_jv_j$. This fits into the exact sequence
\ben
0\to\widehat{\mf{g}}\stackrel{\iota}{\longrightarrow}\R^2\stackrel{\beta}{\longrightarrow}\R\to0
\ee
which descends to the sequence on torus level
\ben
\xymatrix@R=1pt{
0 \ar[r] & \widehat G \cong\quotient{\R}{\Z} \ar[r]^{\bar\iota} & (\quotient{\R}{\Z})^2 \ar[r]^{\bar\beta} & \quotient{\R}{\Z} \ar[r] & 0.\\
& t \ar@{|->}[r] & (t,kt)
}
\ee
Now in general one can show that $\xi=\iota^*(\eta_1,\eta_2)$ is a regular value for the induced moment map $\mu_{\widehat G}=\iota^*\mu_0$. In our case, $\mu_{\widehat G}^{-1}(1)\cong S^3$ on which $\widehat G\cong S^1$ acts via the inclusion $\bar\iota$ (written multiplicatively) $t\in S^1\mapsto (t,t^k)\in T^2$. By the {\em symplectic reduction principle} (see for instance Section 23 in~\cite{cds} and also the example after Definition~\ref{sympform}), $S^3/S^1$ is a symplectic orbifold which inherits a toric structure from the action of $T^2/\widehat G$. In particular, the unlabelled polytope (where $k=1$) gives $U_\Delta=\PP^1$ with the Fubini--Study form. The general construction (without proof) will be outlined in Section~\ref{sympstack}. Note that the point corresponding to $(0,1)$ in $S^3/S^1$ is the only one with non--trivial stabiliser group (which is isomorphic to the group of k--th roots of unity $\Z_k\subset S^1$). As an orbifold, $S^3/S^1$ is the so--called $k$--{\bf conehead} (an explicit orbifold atlas will be  exhibited in the next section).
\end{ex}

In the symplectic category labelled polytopes thus occur rather naturally. As we have mentioned before and seen in the previous example, the labels give rise to codimension $1$ singularities. This cannot happen for algebraic toric varieties coming from a fan -- they are necessarily {\em normal} and have thus singularities of codimension at least $2$. This is where the idea of a {\em stacky fan} -- due to Borisov et al~\cite{bcs} -- comes in. In the next few sections we will explain how theses concepts are related.
%
%
%
%
%
\section{Lie groupoids and stacks}\label{groupstack}
In order to compare the results of~\cite{lt} and~\cite{bcs} we need to pass from orbifolds to {\em (differentiable) stacks}. These can be thought of either as categories fibred into groupoids~\cite{bx} or as pseudofunctors from the category of manifolds to the category of groupoids~\cite{h} which in both cases satisfy additional gluing conditions. We stick to the former approach, but to keep the exposition elementary we will not give a complete definition of a stack. Instead, we rather emphasise their description by means of (equivalence classes of) {\em Lie groupoids} using the dictionary established in~\cite{bx}. 
\paragraph{Stacks.}
In the following we will consider the category $\mc{M}$ of (differentiable) manifolds (not necessarily Hausdorff) with (differentiable) maps as morphisms. Good references for this section are the aforementioned texts by Behrend and Xu~\cite{bx} and Heinloth~\cite{h}. A short introduction to the idea of a stack which is sufficient for our purposes is given in~\cite{f}. 

\begin{dfn}\label{groupoid}
{\rm (i)} A category is called a {\bf groupoid} if every morphism is invertible.

{\rm (ii)} A {\bf category fibred in groupoids (CFG)} over $\mc{M}$, written $\mc{X}\to\mc{M}$, is a category $\mc{X}$ together with a functor $\pi:\mc{X}\to\mc{M}$ satisfying the following property. For every morphism $U\to V$ in $\mc{M}$ and every object $y$ of $\mc{X}$ lying over $V$ (i.e.\ $\pi(y)=V$), there exists a morphism $f:x\to y$ lying over $U\to V$ (i.e.\ $\pi(f)=U\to V$) which is unique up to unique isomorphism. This means that for any other morphism $\tilde f:\tilde x\to y$ lying over $U\to V$ there exists a unique isomorphism $\alpha:\tilde x\to x$ lying over the identity of $U$ \st $\tilde f=f\circ\alpha$. A {\bf morphism} between CFGs $\mc{X}$ and $\mc{X}'$ is a functor $\mc{X}\to\mc{X}'$ which commutes with the projections to $\mc{M}$. An {\bf isomorphism} is a morphism $\mc{X}\to\mc{X}'$ which is an equivalence of categories.
\end{dfn}

\begin{rem}
(i) One often refers to the $x$ of the definition as the {\em pull--back} of $y$ via $U\to V$. It is unique up to unique isomorphism. It follows that for a manifold $U$, the subcategory $\mc{X}(U)$ of $\mc{X}$ consisting of objects lying over $U$ and morphisms lying over the identity, is a groupoid. We call $\mc{X}(U)$ the {\em fibre} of $\pi:\mc{X}\to\mc{M}$.

(ii) Similarly one could consider CGFs over base categories other than $\mc{M}$ such as topological spaces, complex (analytic) spaces \textellipsis. The collection of CGFs over some base category defines itself a $2$--category~\cite{g}.
\end{rem}

\begin{ex}
(i) Let $U$ be a manifold. We define $\underline U$ to be the category whose objects are maps $X\to U$ between manifolds. A morphism between $f:X\to U$ and $g:Y\to U$ is a map $h:X\to Y$ \st $f=g\circ h$. The projection $\pi:\underline U\to\mc{M}$ is given by $\pi(X\to U)=X$ so that $\underline U(X)=C^\infty(X,U)$. The pull--back of $f:Y\to U$ via $g:X\to Y$ is obtained by the usual pull--back of maps $g^*f=f\circ g:X\to U$.

(ii) Let $G$ be a Lie group. We define the stack $BG$ to be the category consisting of objects $(U,P)$ where $p_P:P\to U$ is a principal $G$--fibre over $U$. Morphisms $(U,P)\to (V,Q)$ consist of pair of maps $(f:U\to V,\hat f:P\to Q)$ \st $\hat f$ is $G$--equivariant map and $p_Q\circ\hat f=f\circ p_P$. The projection $\pi:BG\to\mc{M}$ is defined by $(U,P)\mapsto U$. The fibre $BG(U)$ is thus the subcategory of principal $G$--fibre bundles over $U$ with bundle maps as morphisms.

(iii) More generally, let $G$ act on a manifold $U$. We define a CFG $[U/G]$ with fibres
\ben
[U/G](X):=\{(P\to X,u:P\to U)\,|\,P\in BG(X),\,u\mbox{ is $G$--equivariant}\}
\ee
by taking the same morphisms as in (ii) subject to the additional condition that $\hat f$ must form a commutative triangle with the $G$--equivariant maps to $U$. We then recover the previous examples. Indeed, if the action of $G$ is proper and free, then $U/G$ is again a manifold so that any pair $(P\to X,u:P\to U)$ in $[U/G]$ is determined by $\tilde u:P/G\cong X\to U/G$, whence $[U/G]\cong\underline{U/G}$. Secondly, taking a one point space $U=\ast$, then $(P\to X,u:P\to\ast)\in[\ast/G](X)$ is determined by $P\to X$, that is, $[U/G]\cong BG$.
\end{ex}

A {\em stack} is a CFG $\mc{X}\to\mc{M}$ which satisfies certain gluing conditions. The previous examples of CFGs all define stacks (see~\cite{bx} or~\cite{h}). A(n) {\em (iso)morphism between stacks} $\mc{X}$ and $\mc{X}'$ is a(n) (iso)morphism between CFGs. If $U$ is a manifold, then $\mathrm{Mor}_\mc{M}(\underline U,\mc{X})\cong\mc{X}(U)$ where a functor $F:\underline{U}\to\mc{X}$ corresponds to $u=F(\Id_U)\in\mc{X}(U)$ (see Lemma 1.3 in~\cite{h}). 

\begin{ex} 
(i) A morphism $F:\underline U\to\underline V$ is induced by a map $f:U\to V$. Any $g:X\to U$ is mapped to $f\circ g=g^\ast f\in\underline V(X)$.

(ii) A morphism $\underline U\to BG$ is given by a principal fibre bundle $P\to U$. Any $g:X\to U$ is mapped to the pull--back bundle $g^\ast P\in BG(X)$.

(iii) A morphism $\underline U\to [V/G]$ is given by a principal $G$--fibre bundle $P\to U$ together with a $G$--equivariant map $u:P\to V$.  Any $g:X\to U$ is mapped to $(g^\ast P,\hat g^*u)\in [V/G](X)$, where $\hat g$ is the induced bundle map $g^\ast P\to P$ covering $g:X\to U$.
\end{ex}

To define a {\em differentiable} stack we need one more notion.

\begin{dfn}
Given three CFGs $\mc{X}$, $\mc{Y}$, and $\mc{Z}$ over $\mc{M}$ and morphisms $f:\mc{X}\to\mc{Y}$ and $g:\mc{Y}\to\mc{Z}$, the {\bf fibre product} $\mc{X}\times_\mc{Z}\mc{Y}$ is defined to be the following category. Objects are triples
$(x,y,\alpha)$, where $x$ and $y$ are objects in $\mc{X}$ and $\mc{Y}$ lying over the same object $U$ in $\mc{M}$ and $\alpha:f(x)\to g(y)$ is an isomorphism in $\mc{Z}$ lying over the identity of $U$. A morphism from
$(x',y',\alpha')$ to $(x,y,\alpha)$ is given by morphisms $a:x'\to x$ in $X$ and $b:y'\to y$ in $Y$ lying over the
same morphism $U'\to U$ in $\mc{M}$ such that $\alpha\circ f(a)=g(b)\circ\alpha':f(\alpha')\to g(y)$. 
\end{dfn}

With the obvious projection $\mc{X}\times_\mc{Z}\mc{Y}\to\mc{M}$, the fibre product becomes itself a CFG over $\mc{M}$.
 
\begin{dfn}\label{repstack}
{\rm(i)} A stack $\mc{X}$ is {\bf representable} if it is isomorphic to a stack $\underline U$ for some manifold $U$. 

{\rm(ii)} A stack $\mc{X}$ is {\bf differentiable} if there exists a morphism $\underline X\to\mc{X}$ \st for any morphism $\underline U\to\mc{X}$, the resulting fibre product $\underline X\times_\mc{X}\underline U$ is representable and the natural map between manifolds induced by the morphism $\underline X\times_\mc{X}\underline U\to\underline U$ is a surjective submersion. The morphism $\underline X\to\mc{X}$ is said to be an {\bf atlas} of $\mc{X}$.
\end{dfn}

\begin{rem}
In terms of algebraic geometry, a submersion is essentially a smooth map. Loosely speaking then, a representable morphism is a morphism with  differentiable fibres.
\end{rem}

\begin{ex}
(i) Let $\underline U\to\underline W$ and $\underline V\to\underline W$ be morphisms induced by submersions $U\to W$ and $V\to W$ (this implies in particular that $U\times_WV$ is again a manifold). As a consequence of the universal property of the (set--theoretic) fibre product, $\underline U\times_{\underline W}\underline V$ is isomorphic to $\underline{U\times_WV}$. Since the induced map $U\times_WV\to V$ is clearly a surjective submersion, the stack $\underline U$ is differentiable. An atlas is provided by $\Id:U\to U$. Any manifold can therefore be considered as a differentiable stack in a natural way. We sometimes simply write $U$ for $\underline U$ if there is no risk of confusion.

(ii) The morphism $\ast\to BG$ represented by $\ast\times G\in BG(\ast)$ is an atlas for $BG$. Indeed, consider a morphism $\underline V\to BG$ associated with $P\to V$ in $BG(V)$. Then
\begin{eqnarray*}
(\ast\times_{BG}\underline V)(X) & = & \{(f:X\to \ast,g:X\to V,\alpha:X\times G\cong g^*P\}\\
& \cong & \{(g:X\to V,\sigma:X\to g^*P)\,|\,p_{g^\ast P}\circ\sigma=\Id_X\}\\
& \cong & C^\infty(X,P)\\
& = &\underline P(X). 
\end{eqnarray*}

(iii) Finally, consider an action $\bullet:U\times G\to G$ of the Lie group $G$ on $U$. An atlas of $[U/G]$, the so--called {\bf quotient stack} of $U$ and $G$, is provided by the morphism $\underline U\to[U/G]$ corresponding to $(U\times G\to U,\bullet:U\times G\to U)$. Indeed, a calculation similar to (ii) shows that for a morphism $\underline V\to[U/G]$ represented by $(P\to V, u:P\to U)$, one has $\underline U\times_{[U/G]}\underline V\cong\underline P$.
\end{ex}

\begin{rem}
Differentiable stacks form a full sub--$2$--category of the $2$--category of CFGs over $\mc{M}$ consisting of differentiable stacks.
\end{rem}
\paragraph{Lie Groupoids.}
Given an atlas $x:\underline X\to\mc{X}$, we can consider the (representable) fibre product
\ben
\mathrm{Iso}(x):=\underline X\times_\mc{X}\underline X
\ee
together with its two canonical morphisms to $\underline X$. Its inherent structure can be axiomatised as follows.

\begin{dfn}
{\rm(i)} Let $\mc{C}$ be a category. A {\bf groupoid object} in $\mc{C}$ or {\bf groupoid} for short consists of two manifolds $R$ and $U$ and five structure maps, $s:R\to U$ (the {\em source map}), $t:R\to U$ (the {\em target map}), $i:R\to R$ (the {\em inverse map}), $1:U\to R$ (the {\em unit map}) and $m:R\times_UR=\{(g,h)\in R\times R\,|\,s(g)=t(h)\}\to R$ (the {\em multiplication map}). We usually write $i(g)=g^{-1}$, $1(x)=e_x$ and $m(g,h)=gh$. For any $k$, $h$ and $g$ in $R$ these maps are required to satisfy
\begin{itemize}
	\item $s(gh)=s(h)$, $t(gh)=t(g)$,
	\item $(gh)k=g(hk)$ whenever defined,
	\item $e_{t(g)}g=g=ge_{s(g)}$,
	\item $s(g^{-1})=t(g)$, $t(g^{-1})=s(g)$, $g^{-1}g=e_{s(g)}$, $gg^{-1}=e_{t(g)}$.
\end{itemize}

{\rm(ii)} A {\bf Lie groupoid} is a groupoid object in $\mc{M}$ where $s$ (and thus $t$) is a submersion. A {\bf morphism} between two Lie groupoids $R\rightrightarrows U$ and $S\rightrightarrows V$  is a differentiable functor which preserves the groupoid structure. More concretely, it is a pair of smooth maps $\Phi:R\to S$ and $\phi:U\to V$ compatible with the structure maps, i.e.\ for all $g,g'\in R$ and $x\in U$ we have $\phi(s(g))=s(\Phi(g))$, $\phi(t(g))=t(\Phi(g))$, $\Phi(e_x)=e_{\phi(x)}$ and $\Phi(gg')=\Phi(g)\Phi(g')$ whenever this makes sense. We denote this morphism by $(\Phi,\phi)$.
\end{dfn}

\begin{rem}
(i) The condition that $s$ is a submersion implies in particular that $R\times_UR$ is again a manifold.

(ii) If $\mc{G}$ is a small category, then $\mc{G}$ is a groupoid in the sense of Definition~\ref{groupoid} (i) \iff it is a groupoid object for the category of sets (with $s(U\to V)=U$, $t(U\to V)=V$, $i$ taking a morphism to it inverse etc.).
\end{rem}

Schematically we can write a Lie groupoid as
\be\label{groupseq}
R\times_UR\stackrel{m}{\to}R\stackrel{i}{\to}R\begin{array}{c}{\mbox{\scriptsize s}}\\[-5pt]\rightrightarrows\\[-5pt]{\mbox{\scriptsize t}}\end{array}U\stackrel{1}{\to}R.
\ee
In general we simply write $R\rightrightarrows U$ for a Lie groupoid as given by~\eqref{groupseq}. 

\begin{dfn}\label{isogroupoid}
Let $R\rightrightarrows U$ be a Lie groupoid.

{\rm (i)} On $U$ consider the equivalence relation $x\sim y$ \iff there exists $g\in R$ with $s(g)=x$ and $t(g)=y$. The quotient $U/\sim$ is called the {\bf space of orbits} or {\bf coarse moduli space} of the Lie groupoid.

{\rm(ii)} The {\bf isotropy group} of $x\in U$ is the set $R_x:=s^{-1}(x)\cap t^{-1}(x)$ (this is indeed a group \wrt the natural group structure induced by $m$). 
\end{dfn}

\begin{ex}
(i) Every manifold $U$ defines the {\em unit groupoid} $U\rightrightarrows U$ with $s$, $t$, $i$ and $1$ the identity and $p\cdot p=p$ for $p\in U$. The isotropy groups are $U_p=\{p\}$ and the coarse moduli space is $U$. More generally, consider a submersion $X\to U$. Then $X\times_UX$ defines a Lie groupoid with $s$ and $t$ the projections, $1$ the diagonal, the inverse interchanging the two factors and multiplication sending $(x,y)$ and $(y,z)$ to $(x,z)$.

(ii) Every Lie group $G$ can be regarded as a Lie groupoid $G\rightrightarrows\ast$ (with $\ast$ denoting the one point space), where $m$ is usual multiplication, $1(\ast)=e_G$ and $i$ the map taking a group element to its inverse. It follows that $G_\ast=G$ while the coarse moduli space is $\ast$.

(iii) Every (left) $G$--space $U$ gives the {\em translation groupoid} $G\times U\rightrightarrows U$. Here, $s(g,u)=u$, $t(g,u)=gu$, $(g,hx)\cdot(h,x)=(gh,x)$, $i(g,x)=(g^{-1},gx)$ and $1(x)=(e_G,x)$. The isotropy group of $x$ is the stabiliser under the action, i.e.\ the set of pairs $(g,x)$ \st $gx=x$. Further the coarse moduli space is just the space of orbits $U/G$. 

(iv) Let $x:\underline X\to\mc{X}$ be an atlas of a differentiable stack. Then $\mathrm{Iso}(x)\rightrightarrows X$ defines a Lie groupoid. Indeed, $\mathrm{Iso}(x)$ consists of triples $(f:U\to X,g:U\to X,\varphi:x(f)\cong x(g))$ and the canonical projections taking such a triple to $f$ and $g$ respectively define the source and target maps. Multiplication with $(f',g',\psi:x(f')\cong x(g'))$ is defined by $(f,g',\psi\circ\varphi:x(f)\cong x(g'))$. Since $x$ induces a surjective submersion, $\mathrm{Iso}(x)$ can be given a differentiable structure.
\end{ex}

\begin{dfn}
{\rm (i)} A Lie groupoid $R\rightrightarrows U$ is {\bf proper} if the map $(s,t):R\to U\times U$ is proper. 

{\rm (ii)} A Lie groupoid $R\rightrightarrows U$ is {\bf \'etale} if $\dim R=\dim U$, that is, $s$ and $t$ are local diffeomorphisms.
\end{dfn}

Proper \'etale Lie groupoids arise from effective orbifolds as defined in Section~\ref{delzant}. Indeed, let $(U_\alpha,G_\alpha,f_\alpha)$ be an orbifold atlas for $X$. Put $U=\sqcup U_\alpha$ and let $R$ be the set of triples $(x,y,f)$ \st $x$ and $y$ get mapped to the same point in $X$ and $f$ is a germ of a diffeomorphism mapping $x$ to $y$. Then $s(x,y,f)=x$, $t(x,y,f)=y$, $i(x,y,f)=(y,x,f^{-1})$, $1(x_\alpha)=(x_\alpha,x_\alpha,\id_{U_\alpha})$ and $(y,z,f)(x,y,g)=(x,z,f\circ g)$. Moreover, the sheaf topology on $R$ turns $s$ and $t$ into local homeomorphisms which induce a differentiable structure on $R$ for which $s$ and $t$ become local diffeomorphisms. Further, the resulting Lie groupoid $R\rightrightarrows U$ is proper by Proposition 5.29 in~\cite{mm}. Note that the isotropy group of a point $x\in U$ (in the sense of Section~\ref{delzant}) is just $R_x$ as given in Definition~\ref{isogroupoid}. Hence the isotropy groups are discrete and Proposition 5.20 in~\cite{mm} implies that $R\rightrightarrows U$ is also \'etale. 

\begin{ex}
Take $\PP^1=\C\cup\{\infty\}$ and remove a disk $D$ around $\infty$. We obtain the $k$--conehead encountered above by gluing in the cone $D/\Z_k$ of angle $2\pi/k$, see Figure~\ref{fig:narrenkappe}. The resulting space is still homeomorphic to $\PP^1$. An orbifold atlas is given by $(U_0=\C,\{e\},f_0(z)=[z:1])$ and $(U_1=\C,\Z_k,f_1(z)=[1:z^k])$. Indeed, $z$, $w\in U_1$ get mapped to the same point \iff $w=e^{2\pi l/k}z$ for some $l=0,\ldots,k-1$ so that $e^{2\pi l/k}$ induces the required germ of diffeomorphisms. On the other hand, if $z\in U_0$ and $w\in U_1$ get mapped to the same point, a germ is induced by $f_{01}(w)=w^{-k}$.
\end{ex}

\begin{figure}
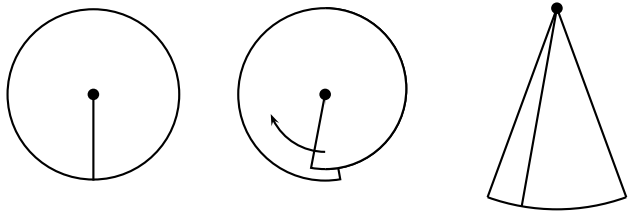

\narrenkappe
\caption{A cone of angle $2\pi/k$}
\label{fig:narrenkappe}
\end{figure}
\paragraph{Morita equivalence.}
An atlas $x:\underline X\to\mc{X}$ gives rise, as we have seen, to a Lie groupoid $\mathrm{Iso}(x)\rightrightarrows X$. Conversely, one can associate with a Lie groupoid $R\rightrightarrows U$ a differentiable stack which carries an atlas giving back $R\rightrightarrows U$ up to isomorphism (see for instance~\cite{bx}). Of course, different Lie groupoids can give rise to isomorphic stacks in the same way two different atlases of a topological manifold can give rise to the same differentiable structure. This structure will be indeed the same if we can pass to a common refinement. We will formalise a similar concept for Lie groupoids now.

\begin{dfn}
{\rm(i)} A {\bf Morita morphism} is a Lie groupoid morphism $(\Phi,\phi)$ from $R\rightrightarrows U$ to $S\rightrightarrows V$ which satisfies the following two properties: 
\begin{itemize}
	\item The diagram
\ben
\xymatrix{
R\ar[rr]^{(s,t)} \ar[d]_{\Phi} & & U\times U\ar[d]^{\phi\times\phi}\\
S\ar[rr]^{(s,t)} & & V\times V
}
\ee
is cartesian, i.e.\ $R$ is isomorphic to the fibred product $S\times_{V\times V}(U\times U)$. 
	\item The map $t\circ\mr{pr}_1:S\times_{s,V,\phi}U\to V$ sending $(h,y)$ to $t(h)$, is a surjective submersion. 
\end{itemize}

{\rm(ii)} Two Lie groupoids $R\rightrightarrows U$ and $S\rightrightarrows V$ are {\bf Morita equivalent} if there exists a third Lie groupoid $T\rightrightarrows W$ with Morita morphisms to $R\rightrightarrows U$ and $S\rightrightarrows V$.
\end{dfn}

\begin{rem}
(i) Here we followed the terminology of~\cite{bx};~\cite{mm} speak of weak equivalences and weakly equivalent respectively.

(ii) Morita equivalence defines an equivalence relation between Lie groupoids (see the remark after Proposition~5.12 in~\cite{mm}).
\end{rem}

Then we have (cf.\ Theorem 2.26 in~\cite{bx})

\begin{prp}
Let $\mc{X}$ and $\mc{Y}$ be differentiable stacks which are associated with the Lie groupoids $R\rightrightarrows U$ and $S\rightrightarrows V$. Then the stacks $\mc{X}$ and $\mc{Y}$ are equivalent \iff$R\rightrightarrows U$ and $S\rightrightarrows V$ are Morita equivalent.
\end{prp}

We will sometimes abuse language and refer to a Lie groupoid associated with a stack as an atlas of the stack.

\begin{rem}
One can turn the category Lie groupoids into a $2$--category and establish a dictionary between differentiable stacks and Lie groupoids, see~\cite{bx}.
\end{rem}

\begin{ex}
(i) Consider the stack $\underline U$ with atlas provided by the identity $\Id:U\to U$. Then $\mathrm{Iso}(\Id)\rightrightarrows U$ is the unit groupoid defined in the examples after Definition~\ref{isogroupoid}. 

(ii) Consider the quotient stack $[U/G]$. The Lie groupoid $\mathrm{Iso}(U\times G,\bullet)\rightrightarrows U$ given by the atlas induced by $(\mathrm{pr}_1:U\times G\to U,\bullet:U\times G\to U)$ is the translation groupoid $G\times U\rightrightarrows U$. Indeed, objects in $\underline U\times_{[U/G]}\underline U(X)$ are determined by triples $(f:X\to U,g:X\to U,\varphi:X\to G)$ with $\varphi(x)\bullet f(x)=g(x)$ (cf.\ the example after Definition~\ref{repstack} to see this). Sending this triple to $(\varphi,f)$ gets the isomorphism $\underline U\times_{[U/G]}\underline U\cong\underline G\times\underline U$. With the second projection as source map and group action as target map (sending $(\varphi,f)$ to $\varphi\bullet f:X\to U$) we obtain the translation groupoid. In particular, taking $U=\ast$ we recover the Lie groupoid $\mathrm{Iso}(\ast\times G)$ of $BG$ coming from the natural atlas $\ast\to BG$, which is just the Lie groupoid $G\rightrightarrows\ast$.
\end{ex}

In the sequel we shall be mainly interested in the case of quotient stacks. The definition of Morita equivalence takes an easier shape when applied to this special case. First, a morphism $[U/G]\to[V/H]$ consists of a morphism $\phi:U\rightarrow V$ and a group homomorphism $\psi:G\rightarrow H$ which are compatible in the following sense. If $\bullet_G$ and $\bullet_H$ denote the respective group actions of $G$ and $H$ on $U$ and $V$, we require that
\ben
\phi(g\bullet_Gu) = \psi(g)\bullet_H\phi(u)
\ee
for all $g\in G$ and $u\in U$. Then $(\psi,\phi)$ is a Morita equivalence if
\begin{enumerate}
\item[\refmoritadiag] the diagram
\ben
\xymatrix{
G \times U \ar[rr]^{(\pr_2,\bullet_G)} \ar[d]_{\psi \times \phi} & & U \times U \ar[d]^{\phi \times \phi}\\
H \times V \ar[rr]^{(\pr_2,\bullet_H)} & & V \times V
}
\ee
is cartesian and
\item[\refmoritasurj] the morphism 
\ben
\Bhalfarray{ccccc}
H\times U & \rightarrow & V \\
(h,m) & \mapsto & h\bullet_H\phi(m)
\Ehalfarray
\ee
is surjective.
\end{enumerate}

\begin{rem}
If both $\phi:U\hookrightarrow V$ and $\psi:G\hookrightarrow H$ are inclusions, then~\refmoritadiag{} is tantamount to condition
\begin{flalign*}
\text{\refmoritadiagp} && \forall h\in H \text{ and }m\in U:\ \left[h\bullet m\in U\Rightarrow h \in G \right] . &&
\end{flalign*}
\end{rem}
%
%
%
%
%
\section{Symplectic toric DM stacks}\label{sympstack}
In this section we discuss the extension of the Lerman--Tolman theorem~\ref{lertol} to the stack setting.
\paragraph{Deligne--Mumford stacks.}
\begin{dfn}
A {\bf Deligne--Mumford stack} or {\bf DM stack} for short is a differentiable stack which admits an atlas given by a proper \'etale Lie groupoid $R\rightrightarrows U$. Its {\bf dimension} is $\dim\mc{X}=\dim U=\dim R$.
\end{dfn}

The dimension is indeed well--defined, cf.\ Section 2.5 in~\cite{bx}.

\begin{ex}
(i) The unit Lie groupoid $U\rightrightarrows U$ is \'etale and proper. Hence every manifold considered as a stack is a DM stack.

(ii) As discussed in Section~\ref{groupstack}, an effective orbifold in the sense of Section~\ref{delzant} gives rise to a proper \'etale Lie groupoid and thus to a DM stack. 
\end{ex}

To define further geometric structure on a DM stack we take the viewpoint of~\cite{lm} and define objects \wrt a fixed (not necessarily \'etale) atlas $R\rightrightarrows U$. To check independence of the atlas one can either show invariance under Morita equivalence or compatibility of these definitions with abstract stack theory.  
\paragraph{Symplectic DM stacks.}
For the definition of vector fields and differential forms on DM stacks we first generalise the concept of a Lie algebra associated with a Lie group to Lie groupoids.

\begin{dfn}
A {\bf Lie algebroid} over a manifold $M$ is a vector bundle $\mf{a}\to M$ with a Lie bracket on its space of sections $C^\infty(\mf{a})$, together with a vector bundle morphism $a:\mf{a}\to TM$ called the {\bf anchor} \st
\begin{itemize}
	\item the induced map $C^\infty(\mf{a})\to\mf{X}(M)$ between sections of $A$ and vector fields on $M$ is a Lie algebra morphism.
	\item for all $v,\,w\in C^\infty(\mf{a})$ and $f\in C^\infty(M)$, the identity $[v,fw]=f[v,w]+df(a(v))w$ holds.
\end{itemize}
\end{dfn}

With a Lie groupoid $R\rightrightarrows U$ we can canonically associate a Lie algebroid $\mf{r}\to U$ as follows (cf.~\cite{mm} Section 6.1). For an arrow $h:y\to x$ we can define a ``left multiplication'' $L_h:t^{-1}(y)\to t^{-1}(x)$ by composition $L_h(g):=hg$. This lifts to the involutive vector subbundle $\ker ds\to R$ of $TR\to R$. Namely, given $\xi\in(\ker ds)_g$ for some $g:z\to y$ we define $h\xi:=dL_h(\xi)$ which lies in $(\ker ds)_{hg}$. A section $X\in C^\infty(\ker ds)$ is {\em invariant} if $X(hg)=hX(g)$. The invariant sections form a Lie subalgebra $A$ of $C^\infty(\ker ds)$. As sections of $A$ are determined by their restriction to the set of units, we get a linear isomorphism between $A$ and the space of sections of the vector bundle $\mf{r}:=1^*\ker ds\to U$. In particular, $\mf{r}$ inherits a natural Lie algebra structure. We take $a=dt:\mf{r}\to TU$ as anchor map and call $(\mf{r},a)$  the {\bf associated} Lie algebroid.

\begin{ex}
Consider a Lie group $G$ as Lie groupoid via $G\rightrightarrows\ast$. The associated Lie algebroid $\mf{g}=1^\ast\ker ds\to\ast$ is a vector bundle whose Lie algebra structure is precisely the Lie algebra structure of left--invariant vector fields on $G$. We thus recover the usual Lie algebra of $G$.
\end{ex}

An arbitrary atlas of a DM stack is not necessarily \'etale as this property is not preserved under Morita equivalence. However, the following statement allows us to characterise the Lie groupoids representing a DM stack in terms of their associated Lie algebroids.

\begin{thm}{\bf\cite{cm}}
A Lie groupoid $R\rightrightarrows U$ is Morita equivalent to an \'etale groupoid \iff the Lie algebroid associated with $R\rightrightarrows U$ has injective anchor.
\end{thm}

For a Lie groupoid $R\rightrightarrows U$ representing a DM stack we can therefore think of its associated Lie algebroid as a subbundle of $TU$. This makes DM stacks a convenient class to work with. In the sequel, we consider a DM stack $\mc{X}$ together with a fixed atlas $R\rightrightarrows U$. 

\medskip

We first define the {\em space of $k$--forms} on $\mc{X}$ by
\ben
\Omega^k(\mc{X}):=\{(\alpha_1,\alpha_0)\in\Omega^k(R)\times\Omega^k(U)\,|\,s^*\alpha_0=\alpha_1=t^*\alpha_0\}.
 \ee
We can regard $\Omega^k(\mc{X})$ as $k$--forms annihilating $\mf{r}\subset TU$. If $\mf{r}^\perp\subset T^*U$ denotes the annihilator of $\mf{r}$, then $\alpha_0$ is a section of $\Lambda^k\mf{r}^\perp$. Furthermore, $\alpha_0$ is invariant under the natural ``action'' of $R$ on $U$, where $g\in R$ sends $s(g)$ to $t(g)$. Note that the exterior derivative $d$ commutes with pullbacks. Hence the exterior derivative induces a well--defined map $d:\Omega^k(\mc{X})\to\Omega^{k+1}(\mc{X})$ sending $(\alpha_1,\alpha_0)$ to $(d\alpha_1,d\alpha_0)$. In particular, we can speak about {\em closed} forms, i.e.\ forms in the kernel of~$d$. By Proposition 2.9 (ii) in~\cite{lm}, the resulting de Rham complex $\Omega^\ast(\mc{X})$ does not depend on the chosen atlas up to isomorphism.

\medskip

Analogously we define the space of vector fields $\mf{X}(\mc{X})$. A vector field is a section of $TU/\mf{r}$ which is equivariant under the ``action'' of $R$ on $U$. Concretely, call a pair $(v_1,v_0)$ in $\mf{X}(R)\times\mf{X}(U)$ {\em compatible} if $ds(v_1)=v_0\circ s,\,dt(v_1)=v_0\circ t$. Then we regard two pairs of compatible vector fields $(v_1,v_0)$ and $(w_1,w_0)$ as {\em equivalent} (denoted by $\sim$) \iff they differ only by a compatible pair $(u_1,u_0)$ with $u_1\in(\ker ds+\ker dt)$. (Note that $s$ is a surjective submersion, hence the relation $ds(v_1)=v_0\circ s$ determines $v_1$ up to sections of $\ker ds$, and similarly for $t$.) We then define
\ben
\mf{X}(\mc{X}):=\{(v_1,v_0)\in\mf{X}(R)\times\mf{X}(U)\,|\,ds(v_1)=v_0\circ s,\,dt(v_1)=v_0\circ t\}/\sim.
\ee
To keep notation simple we denote by $(v_1,v_0)$ both the compatible pair and the induced equivalence class. Again, up to isomorphism, $\mf{X}(\mc{X})$ does not depend on the chosen atlas (Proposition 2.9 (i) in~\cite{lm}). 
 
\begin{ex}
(i) For a manifold $U$ take the associated Lie groupoid $\mathrm{Iso}(\Id)\rightrightarrows U$ considered in Section~\ref{groupstack}. Then we recover the usual notion of a differential form and a vector field.

(ii) If $G$ acts freely and properly on $U$, let $\mf{g}^\sharp\to U$ denote the subbundle of fundamental vector fields in $TU$. Then $[U/G]\cong U/G$, whence $\Omega^k([U/G])=C^\infty(\Lambda^k\mf{g}^{\sharp\perp})^G$, the $G$--invariant $k$--forms which annihilate $\mf{g}^\sharp$. Similarly, $\mf{X}([U/G])=C^\infty(TU/\mf{g}^\sharp)^G$, the $G$--invariant sections of $TU/\mf{g^\sharp}$.
\end{ex}

From the previous definitions it follows that the {\em contraction} of a vector field $v=(v_1,v_0)$ with a form $\alpha=(\alpha_1,\alpha_0)$,
\ben
v\llcorner\alpha:=(v_1\llcorner\alpha_1,v_0\llcorner\alpha_0)
\ee
(where $v_j\llcorner\alpha_j$ is the usual contraction $\Omega^k\to\Omega^{k-1}$) is a well--defined operation $\Omega^k(\mc{X})\to\Omega^{k-1}(\mc{X})$. A $2$--form $\omega$ on a DM stack $\mc{X}$ is said to be {\em non--degenerate} \iff contraction with $\omega$ induces a linear isomorphism $\mf{X}(\mc{X})\to\Omega^1(\mc{X})$. 

\begin{dfn}\label{sympform}
A $2$--form $\omega$ on a DM stack is called {\bf symplectic} if it is non--degenerate and closed. A DM stack $(\mc{X},\omega)$ together with a symplectic form is called a {\bf symplectic DM stack}.
\end{dfn}

\begin{ex}
In continuation of the previous example (ii) a symplectic form on $[U/G]\cong U/G$ is a closed $G$--invariant $2$--form on $U$ whose kernel is {\em precisely} $\mf{g^\sharp}$ (this is the non--degeneracy condition). As an example, consider the action of a Lie group $G\subset T^d$ on $\C^d$ with associated moment map $\mu_G:\C^d\to\mf{g}^\ast$ (see the example after Definition~\ref{hamac}). Let $\xi\in\mf{g}^*$ be a regular value for $\mu_G$ and assume that $G$ acts freely on the embedded submanifold $i:\mu^{-1}_G(\xi)\hookrightarrow\C^d$. Then the kernel of the closed $2$--form $\omega=i^*\omega_0$ is $\mf{g}^\sharp$ so that $\omega$ descends to a symplectic form on $\mu^{-1}_G(\xi)/G$. This is the symplectic reduction principle (cf.\ for instance Section 23 in~\cite{cds}) which underlies the construction after Theorem~\ref{symtor}.
\end{ex}
\paragraph{Hamiltonian group actions.}
Next we wish to consider actions by a Lie group on a differentiable stack. Their definition is more subtle than in the case of manifolds for stacks are categories and thus group elements act as functors. However, the composition of two such functors may differ from the functor of the product of the corresponding group elements by a natural transformation. A precise definition of $G$--actions as used here is due to Romagny~\cite{r}. We will not give it here; instead, we rephrase $G$--actions on DM stacks in terms of Lie groupoids (see Prop.\ 1.5 in~\cite{r} and Prop.\ 3.2 in~\cite{lm}).

\begin{prp}
Suppose we have a $G$--action on a DM stack $\mc{X}$ for some Lie group $G$. Then there exists a $\boldsymbol{G}$--{\bf atlas} for $\mc{X}$, that is, there exists a Lie groupoid $R\rightrightarrows U$ where $G$ acts on both $R$ and $U$ freely and compatibly with the structure maps.
\end{prp}

\begin{rem}
If in addition $G$ acts properly on $R$ and $U$, $R/G$ and $U/G$ are again manifolds. The Lie groupoid $R/G\rightrightarrows U/G$ is then an atlas for the differentiable stack $\mc{X}/G$, the {\bf quotient of $\mc{X}$ by} $G$. Of course, $\underline U/G\cong U/G$.
\end{rem}

Let $R\rightrightarrows U$ be a $G$--atlas of $\mc{X}$. For $v\in\mf{g}$, let $v^\sharp=(v_1^\sharp,v_0^\sharp)$ denote the induced fundamental vector field on $R\times U$. Since the action of $G$ commutes with the structure maps, $v^\sharp\in\mf{X}(\mc{X})$.

\begin{dfn}
A $G$--action on a symplectic DM stack $(\mc{X},\omega)$ is called {\bf hamiltonian} if there is a $G$--atlas $R\rightrightarrows U$ with a $G$--equivariant $\mf{g}^*$--valued function $\mu=(\mu_1,\mu_0)$, i.e.\ $\mu\in\Omega^0(\mc{X})\otimes\mf{g}^*$, \st
\ben
v^\sharp\llcorner\omega=\big(d\langle\mu_1,v\rangle,d\langle\mu_0,v\rangle\big)
\ee
for any $v\in\mf{g}$. Again we refer to $\mu$ as the {\bf moment map} of the action.
\end{dfn}
\paragraph{Symplectic reduction.}
The next definition is also taken from~\cite{lm}.

\begin{dfn}\label{symtoricstack}
A {\bf symplectic toric DM stack} is a symplectic DM stack $(\mc{X},\omega)$ with a hamiltonian action by a compact torus $T$ \st
\begin{itemize}
	\item $T$ acts effectively on the coarse moduli space for any given $T$--atlas $R\rightrightarrows U$.
	\item $\dim\mc{X}=2\dim T$.
\end{itemize}
\end{dfn}

Generalising the example after Definition~\ref{sympform} we get for any regular value $\xi\in\mf{g}^\ast$ of $\mu_{T^d}$ the symplectic toric DM stack $[\quotient{\mu^{-1}_{T^d}(\xi)}{T^d}]$. Indeed, Theorem 5.4 in~\cite{lm} gives the following.

\begin{thm}{\bf (Lerman--Malkin~\cite{lm})}\label{symtor}
Let $G\subset T$ be a closed subgroup and $1\to\Gamma\to\widehat T\to T\to1$ an extension of the standard compact torus by a finite group $\Gamma$. Let $\widehat G$ denote the corresponding group in $\widehat T$ and $\xi\in\mf{g}^\ast$ be a regular value for the moment map $\mu_{\widehat{G}}:\C^d\to\mf{g}^\ast$. Then the quotient stack
\ben
\C^d\sslash_\xi\widehat G:=\left[\quotient{\mu^{-1}_{\widehat G}(\xi)}{\widehat G}\right]
\ee
is a symplectic toric DM stack acted on by the torus $\widehat{T}/\widehat{G}=T/G$. 
\end{thm}

\paragraph{Examples.}
One way of producing the data of Theorem~\ref{symtor} is to consider a labelled rational polytope $\underline\Delta\subset\R^{d\ast}$ as in Theorem~\ref{lertol}. Using the notational conventions of Section~\ref{delzant}, we define a linear map $\beta:\R^m\to\R^d$ by $\beta(e_j)=n_j u_j$, where $e_1,\ldots,e_m$ is the standard basis of $\R^m$. Since $u_j\in N$ for $j=1,\ldots,m$, the exact sequence
\ben
0\to\ker\beta\stackrel{\iota}{\longrightarrow}\R^m\stackrel{\beta}{\longrightarrow}\R^d\to0
\ee
gives rise to the (additively written) exact sequence on the torus level
\ben
0\to\widehat G:=\ker\bar\beta\stackrel{\bar\iota}{\longrightarrow}T^m\cong(\quotient{\R}{\Z})^m\stackrel{\bar\beta}{\longrightarrow}T^d\cong(\quotient{\R}{\Z})^d\to0,
\ee
that is,
\ben
\widehat G=\{[x]=[(x_1,\ldots,x_m)]\in T^m\,|\,x\in\R^m\mbox{ with }\beta(x)\in N\}.
\ee
Note that $\ker\beta$ is just the Lie algebra of $\widehat G$. Instead of $\beta$ we can also consider $\beta_0:\R^m\to\R^d$ defined by $\beta_0(e_j)=u_j$ which in the same way gives rise to a subgroup $G\subset T^m$. Then there is a finite extension given by the split sequence
\ben
0\to\Gamma\to\widehat G\stackrel{\bar n}{\to}G\to 0
\ee
induced by the map $\bar n[x]=[(n_1x_1,\ldots,n_mx_m)]$. Finally, $\xi=\iota^*(\eta_1/n_1,\ldots,\eta_m/n_m)$, where the $\eta_j$ are determined by~\eqref{etadata}, is a regular value for $\mu_{\widehat G}$ by Theorem 8.1 in~\cite{lt}. Now Theorem~\ref{symtor} applies.

\medskip

Next we consider two concrete examples. The first one is induced by the moment polytope of $\PP^2$ (see Figure~\ref{fig:proj}), but with two different non--trivial labellings. The resulting coarse moduli space has a singular divisor and thus codimension $1$ singularities. The second example comes from a rational, non--smooth polytope with trivial labelling (see Figure~\ref{fig:projweight}). Here, the singularities have codimension $2$.

\medskip

{\bf (i) The projective plane}
\begin{figure}
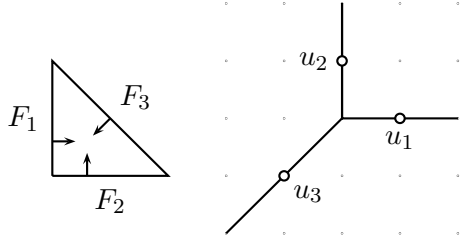

\figproj
\caption{The polytope and its normal fan of the projective plane $\PP^2$.}
\label{fig:proj}
\end{figure}
We have the facets $F_1$, $F_2$ and $F_3$ (see Figure~\ref{fig:proj}) which we label by $(1,1,2)$ and $(2,2,2)$. It follows that $\eta_1=\eta_2=0$ and $\eta_3=1$. The resulting maps $\beta_1$ and $\beta_2$ are given by 
\ben
\Bhalfarray{rclcrcl}
\beta_1:\ \R^3 & \rightarrow & \R^2 &\quad&
\beta_2:\ \R^3 & \rightarrow & \R^2 \\
e_1 & \mapsto & 1 \cdot u_1 &\quad&
e_1 & \mapsto & 2 \cdot u_1 \\
e_2 & \mapsto & 1 \cdot u_2 &\quad&
e_2 & \mapsto & 2 \cdot u_2 \\
e_3 & \mapsto & 2 \cdot u_3 &\quad&
e_3 & \mapsto & 2 \cdot u_3 \\
\Ehalfarray
\ee
which we represent by the matrices
\be\label{matrix1}
M_1=
\begin{pmatrix}
1&0&-2\\
0&1&-2\\
\end{pmatrix}
\textnormal{ and }
M_2=
\begin{pmatrix}
2&0&-2\\
0&2&-2\\
\end{pmatrix}.
\ee
Now $\beta_{1,2}(x_1,x_2,x_3)\in\Z^2$ \iff $x_{1,2}-2x_3\in\Z$ (for $\beta_1$) and $2(x_{1,2}-x_3)\in\Z$ (for $\beta_2$). Since $\bar n_1[x_2,x_2,x_3]=[x_1,x_2,2x_3]$ and $\bar n_2[x_2,x_2,x_3]=[2x_1,2x_2,2x_3]$, we get the exact sequences
\ben
\begin{array}{c}
0\to\Gamma_1\cong\Z_2\to\widehat G_1=\{[2x,2x,x]\,|\,x\in\R\}\stackrel{\bar n_1}{\longrightarrow}G=(\quotient{\R}{\Z})^3\to0\\[5pt]
0\to\Gamma_2\cong\Z_2\to\widehat G_2=\{[x+a,x+b,x]\,|\,x\in\R,\,a,b\in\tfrac{1}{2}\Z\}\stackrel{\bar n_2}{\longrightarrow}G=(\quotient{\R}{\Z})^3\to0.
\end{array}
\ee
In both cases we have $\Gamma_{1,2}=\{(0,0,c)\,|\,c\in\tfrac{1}{2}\Z\}$. Then $\widehat G_1$ and $\widehat G_2$ act on $\C^3$ via the inclusions (written multiplicatively)
\ben
\Bhalfarray{rclcrcl}
S^1\cong\widehat G_1&\rightarrow&\TT^3&\quad&\Z_2\times\Z_2 \times S^1\cong\widehat G_2&\rightarrow&\TT^3\\
t=[(2x,2x,x)]&\mapsto& (t^2,t^2,t) &\quad&([a],[b],t)=[(x+a,x+b,x)] & \mapsto & \left((-1)^{2a}t,(-1)^{2b}t, t\right).\\
\Ehalfarray
\ee
On the other hand, $\iota_k(\ker\beta_k)$, $k={1,2}$, is spanned by $(2,2,1)$ and $(1,1,1)$ respectively so that
\ben
\begin{array}{c}
\mu_{\widehat G_1}(z_0,z_1,z_2)=\iota_1^\ast\circ\mu(z_0,z_1,z_3)=2|z_0|^2+2|z_1|^2+|z_2|^2\\[5pt]
\mu_{\widehat G_2}(z_0,z_1,z_2)=\iota_2^\ast\circ\mu(z_0,z_1,z_3)=|z_0|^2+|z_1|^2+|z_2|^2.
\end{array}
\ee
Hence $\mu_{\widehat G_k}^{-1}(\iota_k^\ast(\eta))=\mu_{\widehat G_k}^{-1}(2)$ is diffeomorphic to $S^5$ and we obtain the toric symplectic DM stacks $\C^3\sslash_2\widehat G_{1,2}=[S^5/\widehat G_{1,2}]$. The isotropy groups are trivial for $(z_0,z_1,z_2)$ with $|z_0|^2+|z_1|^2<c_{1,2}$ where $c_1=1$ and $c_2=2$, and otherwise isomorphic to $\Z_2$. 

\medskip

{\bf (ii) Weighted projective space}
\begin{figure}
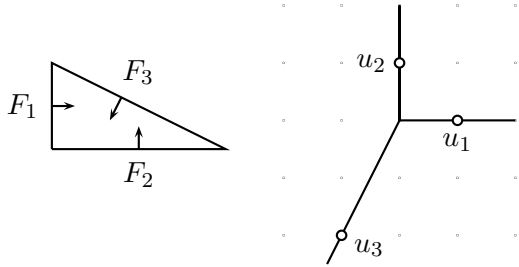

\figweightproj
\caption{The polytope and its normal fan of the weighted projective space $\PP^2(1,1,2)$.}
\label{fig:projweight}
\end{figure}
Here we take the polytope of Figure~\ref{fig:projweight} with trivial labelling, i.e.\ $\beta(e_i)=u_i$, so the associated matrix is 
\be\label{matrix2}
M=
\begin{pmatrix}
1&0&-1\\
0&1&-2\\
\end{pmatrix}.
\ee
Proceeding as above, we see that $\iota(\ker\beta)$ is spanned by $(1,2,1)$ and $\widehat G=\{[(x,2x,x)]\,|\,x\in\R\}$ acts on $\C^3$ via the inclusion
\ben
\Bhalfarray{rcl}
\widehat G\cong S^1 & \rightarrow & \TT^3\\
t & \mapsto & (t,t^2,t).
\Ehalfarray
\ee
Now $\eta=(0,0,2)$, so $2$ is a regular value for $\mu_{\widehat G}(z_0,z_1,z_2)=|z_0|^2+2|z_1|^2+|z_2|^2$ and $\mu^{-1}_{\widehat G}(2)$ is diffeomorphic to $S^5$. We obtain the toric symplectic stack $\C^3\sslash_2\widehat G=[S^5/\widehat G]$ -- the $2$--dimensional analogue of the $2$--conehead. The only nontrivial isotropy group is $\Z_2$ for the image of $(0,1,0)$.
%
%
%
%
%
\section{Complex toric DM stacks}\label{anastack}
A {\bf (complex) toric variety} is a normal variety which contains an algebraic torus $\TT$ as an open dense subset and such that the action of $\TT$ on itself extends to the whole variety. This definition was subsequently generalised by Iwanari~\cite{i} and Fantechi et al.~\cite{fmn} to {\em (complex) toric DM stacks}, that is, (separated) DM stacks together with a stacky DM torus as an open dense subset and such that its action extends to the stack (cf. Definition 3.1 in [loc.~cit.]). Based upon this definition, Fantechi et al carried out a classification of complex toric DM stacks. Prior to this, Borisov et al~\cite{bcs} constructed complex toric DM stacks as quotients. It is this construction we want to outline in this section in the case of trivial generic stabiliser. Note that the construction of~\cite{bcs} via stacky fans gives stacks isomorphic to the toric DM stacks as considered in~\cite{fmn}. However, the isomorphism is not unique if there is a non--trivial generic stabiliser (cf. Theorem II and Remark 7.26 in \cite{fmn}).
\paragraph{Stacky fans.}
For a fan $\Sigma$ we denote by $\Sigma(i)$ the set of $i$--dimensional cones in $\Sigma$.
In particular, $\Sigma(1)$ is the set of rays. Let $m$ be the cardinality of $\Sigma(1)$.

\begin{dfn}
A {\bf stacky fan} is a triple $\Stacky=(\Sigma,N,\beta)$ where
\begin{itemize}
	\item $N\cong\Z^d$ is a lattice with dual lattice $M=\Hom(N,\ZZ)$.
	\item $\Sigma$ is a complete and simplicial fan in $N_\Q = N\otimes\Q$, i.e.\ the union of all cones covers the whole $N_\Q$, and the generators of the rays of each cone are linearly independent.
	\item $\beta:\Z^m\rightarrow N$ is a map \st if  $u_j$ denotes the primitive generator of $\rho_j\in\Sigma(1)$, then $\beta(e_j)=n_j\cdot u_j$ for some $n_j\in\N_{>0}$. We think of $\beta$ as a choice of lattice generators for the rays in $\Sigma(1)$. 
\end{itemize}
\end{dfn}

\begin{rem}
Instead of completeness it suffices to require that the rays of the fan $\Sigma$ generate $N_\Q$ (cf.\ the analogous construction of toric varieties, e.g.~\cite{cox}). Geometrically this means that there are no torus factors, see for instance Section 5.1 in~\cite{cls} or Remark 7.14 in~\cite{fmn}. Moreover, $N$ can be any finitely generated abelian group.
\end{rem}
\paragraph{The case of toric varieties.}
A toric variety $X=X(\Sigma)$ with simplicial fan $\Sigma$ and without torus factors can be written as a good geometric quotient
\ben
X(\Sigma)=\C^m_\Sigma \slash H.
\ee
Here, $\C^m_\Sigma$ and $H$ are defined as follows. Consider $\C^m$ as the direct product of copies of $\C=\Spec k[x_j]$ for every ray $\rho_j\in\Sigma(1)$. The monomials
\ben
x(\sigma):=\prod_{\rho_j \not= \sigma} x_j\textnormal{ for } \sigma \in \Sigma
\ee
generate the so--called {\em irrelevant ideal} $I$ whose vanishing locus is $Z$ (since $\Sigma$ is assumed to be complete, it is actually enough to take only the monomials $x(\sigma)$ associated with the top--dimensional cones $\sigma\in\Sigma(n)$). Hence $Z$ is simply a union of coordinate hyperplanes. We set $\C^m_\Sigma := \C^m \setminus Z$. Let $\TT=\Hom(M,\C^\ast)$ be the algebraic torus of $X$. To obtain the group action we look at one of the most important exact sequences in toric geometry:
\be\label{ex:theSeq}
\Bhalfarray{ccccccccccccc}
0\longrightarrow & M& \longrightarrow & \displaystyle{\bigoplus_{j=1}^m}\ \Z \cdot D_j  & \longrightarrow& \Cl(X)& \longrightarrow0,\\
& w & \mapsto & \sum_j\langle w,u_j\rangle D_j
\Ehalfarray
\ee
where $D_j$ is the $\TT$--invariant divisor corresponding to the ray $\rho_j$, and $\Cl(X)$ is the divisor class group of $X$. We apply the functor $\Hom_\Z(-,\C^\ast)$ to this sequence. Although the $\Hom$--functor is only left--exact, the sequence
\ben
\Bhalfarray{ccccccccc}
1\longrightarrow&\Hom(\Cl(X),\C^\ast)&\longrightarrow&&(\C^\ast)^m=\Hom(\Z^m,\C^\ast)&\longrightarrow&&\TT&\longrightarrow 1
\Ehalfarray
\ee
is still exact, for $\CC^\ast$ is divisible. Then $H:=\Hom(\Cl(X),\C^\ast)$ acts on $\C^m_\Sigma$ via the natural inclusion into $(\C^\ast)^m$. By Theorem 1.11 in Chapter 5 of~\cite{cls} the quotient $\C^m_\Sigma \slash H$ is a geometric quotient and  isomorphic to $X(\Sigma)$.
\paragraph{The generalisation to toric stacks.}
For a stacky fan $\Stacky$ we proceed similarly and define the stack
\be\label{tvstack}
\XX(\Stacky)=\left[\quotient{\C^m_\Sigma}{H(\beta)}\right].
\ee
Here $\C^m_\Sigma$ is constructed as before, but the definition of $H(\beta)$ is more elaborate.

\medskip

Consider the map $\beta:\Z^m\rightarrow N$. As remarked above, $N$ could be any abelian group so that dualising $\beta$ destroys all torsion information. Although we are dealing here with a torsion--free $N$, we present the more general procedure used in~\cite{bcs}. So, instead of dualising $\beta$ directly we consider the mapping cone in the derived category of $\Z$-modules, i.e.\ the exact triangle
\ben
\xymatrix{
 & & \ZZ^m \ar[r]^{\id} \ar[d]^{\beta} & \ZZ^m\\
\ZZ^m \ar[r]^{\beta} & N \ar[r]^{\id} & N\\
\ZZ^m \ar[r] & N \ar[r] & \Cone(\beta) \ar[r] & 
\ZZ^m[1].
}
\ee
In the general situation that $N$ is not free but a finitely generated abelian group, $N$ needs to be replaced by its free resolution in the diagram above.
Now we dualise to get
\ben
\xymatrix{
 & & (\ZZ^m)^\ast & (\ZZ^m)^\ast\ar[l]_{\id^\ast} \\
(\ZZ^m)^\ast & M\ar[l]_{\beta^\ast}  & M\ar[l]_{\id^\ast} \ar[u]_{\beta^\ast} \\
(\ZZ^m)^\ast & M\ar[l]  & \Cone(\beta)^\ast\ar[l]  & 
(\ZZ^m)^\ast[1]\ar[l] .
}
\ee
Rolling out this triangle by taking cohomology leads to the long cohomology sequence,
but we are only interested in its end
\ben
H^0 \Cone(\beta)^\ast \longrightarrow M \stackrel{\beta^\ast}{\longrightarrow} \left(\ZZ^m\right)^\ast 
\stackrel{\beta^\vee}{\longrightarrow} H^1 \Cone(\beta)^\ast \longrightarrow \Ext^1(N,\ZZ) 
\longrightarrow 0.
\ee
Since $N$ is free, $\Ext^1(N,\Z)$ vanishes, hence $H^1 \Cone(\beta)^\ast$ is just
 $\coker(\beta^\ast)$ and $\beta^\vee:\left(\Z^m\right)^\ast\rightarrow\coker \beta^\ast$ is surjective. 

\medskip

On the other hand, the fan $\Sigma$ is complete so that $\beta:\Z^m\rightarrow N$, though not necessarily surjective, has only finite cokernel. Hence $H^0\Cone(\beta)^\ast=\ker(\beta^\ast)=\coker(\beta)^\ast=0$, which yields the exact sequence
\be\label{ex:theSeqStack}
0\longrightarrow M\stackrel{\beta^\ast}{\longrightarrow}\left(\Z^m\right)^\ast\stackrel{\beta^\vee}{\longrightarrow}\coker(\beta^\ast)\longrightarrow0.
\ee
In essence, this is the sequence from~\eqref{ex:theSeq}. Therefore, applying the functor $\Hom_\Z(-,\C^\ast)$ gives an action of
\ben
H(\beta):=\Hom(\coker(\beta^\ast),\C^\ast).
\ee
on $\C^m_\Sigma$ via the embedding $\Hom(\beta^\vee,\C^\ast)$ into $\Hom((\Z^m)^\ast,\C^\ast)$. This defines the stack $\XX(\Stacky)$ of~\eqref{tvstack}.

\begin{rem}
\label{rmk:stackVar}
(i) Explicitly, the map $\beta^\ast$ is given by
\ben
\Bhalfarray{rclcrcl}
M & \stackrel{\beta^\ast}{\longrightarrow} & \Hom(\Z^m,\Z) &\quad& 
M & \longrightarrow & \displaystyle{\bigoplus_{j=1 }^m}\ \Z \cdot D_j\\
w & \mapsto & \left[ e_j \mapsto \skalp{w,\beta(e_j)} 
= \skalp{w,n_j \cdot u_j} \right] &\quad&
w & \mapsto & \sum_j \skalp{w,u_j} D_j.
\Ehalfarray
\ee
Comparing this with the map in~\eqref{ex:theSeq}, we recover the classical toric case by taking for $\beta$ the map $[e_j\mapsto u_j]$.

(ii) For general $N$, the group $H(\beta)$ is $\Hom(H^1\Cone(\beta)^\ast,\C^\ast)$. The case of torsion--free $N$ corresponds precisely to complex toric DM orbifolds (cf.\ Lemma 7.15(2) in~\cite{fmn}).
\end{rem}
\paragraph{Open substacks.}
For toric varieties, the cones $\sigma$ of maximal dimension in the defining (complete and simplicial) fan $\Sigma$ give open charts $U_\sigma$ of $X(\Sigma)$. For toric stacks Proposition 4.3 in \cite{bcs} yields a similar statement.

\medskip

Let $\sigma$ be such a cone of dimension $d$. We can restrict the map $\beta:\Z^m\rightarrow N$ to $\beta_\sigma:\Z^d\rightarrow N$ such that $\beta_\sigma (\N^d)\otimes\Q=\sigma$. Set $N_{\stcky\sigma}=\im\beta_\sigma$. This is a sublattice of $N$ of finite order, i.e.\ $N(\stcky\sigma)=N/N_{\stcky \sigma}$ is a finite group. Then $\stcky \sigma = (\sigma, N, \beta_\sigma)$ defines an open substack $\XX(\stcky \sigma)$ of $\XX( \stcky \Sigma)$. The important observation of Proposition 4.3 in~[loc.~cit.] is that $\XX(\stcky \Sigma)$ is locally the quotient by a finite group $H(\beta_\sigma)\cong N(\stcky\sigma)$:
\ben
\XX(\stcky \sigma) = \left[ \quotient{\CC^d}{H(\beta_\sigma)} \right].
\ee

\begin{prp}
\label{prp:localStack}
The quotient $X(\stcky \sigma) = \quotient{\CC^d}{H(\beta_\sigma)}$ is isomorphic to
$X(\sigma)$.
\end{prp}

\begin{proof}
Let $\beta_\sigma:\Z^d\rightarrow N$ be given by $e_j\mapsto n_ju_j$ and define $\beta_{\sigma,0}:\Z^d \rightarrow N,e_j\mapsto u_j$. If $\phi:\Z^d\hookrightarrow\Z^d$ is the obvious map \st $\beta=\beta_{\sigma,0}\circ\phi$, then there exists a big commutative diagram
\ben
\xymatrix{
 & \displaystyle{\bigoplus_j} \quotient{\Z^\ast}{n_j\Z^\ast} \ar[r]^{\cong \text{\phantom{hallo}}} & \quotient{\coker \beta_\sigma^\ast}{\coker \beta_{\sigma,0}^\ast}  \\
M \ar@{^(->}[r]^{\beta_\sigma^\ast}  & (\Z^d)^\ast \ar@{->>}[r] \ar@{->>}[u] &
 \coker \beta_\sigma^\ast \ar@{->>}[u]
\\
M \ar@{^(->}[r]^{\beta_{\sigma,0}^\ast} \ar@{=}[u] & (\Z^d)^\ast \ar@{->>}[r] \ar@{^(->}[u]_{\phi^\ast} & \coker \beta_{\sigma,0}^\ast \ar@{^(->}[u]\\
}
\ee
The inclusion in the lower right hand side corner and the isomorphism in the top row can be deduced from the snake lemma. We apply the functor $\Hom(-,\CC^\ast)$ to the exact sequence in the right column and obtain
\ben
0 \longrightarrow \bigoplus_j\Z_{n_j}\longrightarrow H(\beta_\sigma) \longrightarrow H(\beta_{\sigma,0})\longrightarrow 0
\ee
where the $\Z_{n_j}$ denote the cyclic groups of $n_i$--th roots of unity. It is well--known from geometric invariant theory that the morphism $\Spec A\rightarrow\Spec A^G$ is a good categorical quotient if $G$ is a reductive group acting algebraically on the affine variety $\Spec A$, whence
\ben
\quotient{\C^d}{H(\beta_\sigma)}=\Spec\C[x_1, \ldots, x_d]^{H(\beta_\sigma)}.
\ee
Using the previous exact sequence we conclude
\ben
\C[x_1, \ldots, x_d]^{H(\beta_\sigma)}=\left(\C[x_1, \ldots, x_d]^{\bigoplus_j\Z_{n_j}}\right)^{H(\beta_{\sigma,0})}\cong\C[x_1, \ldots,x_d]^{H(\beta_{\sigma,0})} 
\ee
to obtain the isomorphism $X(\stcky\sigma)\cong X(\sigma)$.
\end{proof}

\begin{rem}
Even though $X(\stcky\sigma)$ is isomorphic to $X(\sigma)$ as an affine variety, the torus actions are different. Furthermore, these spaces are also different when considered as orbifolds or stacks $[\C^d/H(\beta_\sigma)]$ and $[\C^d/H(\beta_{\sigma,0})]$: The action of $H(\beta_{\sigma,0})$ is free except on a closed subset of codimension at least $2$, which becomes the singular locus of the quotient $X(\sigma)$. As soon as $n_j >1$, $H(\beta_\sigma)$ does \emph{not} act freely anymore. In particular, this action has $\Z_{n_j}$ as isotropy group for any point in the divisor $\{x_j = 0\}$ of $X(\stcky \sigma)$.
\end{rem}
\paragraph{Examples.}
We revisit the examples from Section~\ref{sympstack}.
\subsubsection*{(i) The projective plane}
For the fan $\Sigma$ of $\PP^2$ (see Figure~\ref{fig:proj}) we consider again the maps $\beta_{1,2}$:
\ben
\Bhalfarray{rclcrcl}
\beta_1:\ \ZZ^3 & \rightarrow & N &\quad&
\beta_2:\ \ZZ^3 & \rightarrow & N \\
e_1 & \mapsto & 1 \cdot u_1 &\quad&
e_1 & \mapsto & 2 \cdot u_1 \\
e_2 & \mapsto & 1 \cdot u_2 &\quad&
e_2 & \mapsto & 2 \cdot u_2 \\
e_3 & \mapsto & 2 \cdot u_3 &\quad&
e_3 & \mapsto & 2 \cdot u_3 \\
\Ehalfarray
\ee
represented by the matrices $M_{1,2}$ in~\eqref{matrix1}. For the computation of $\coker\beta^\ast_{1,2}$  we diagonalise the transposes of $M_{1,2}$ with an element in $\Gl(3,\Z)$ and get
\ben
\begin{pmatrix}
1 & 0 & 0\\
0 & 1 & 0\\
2 & 2 & 1\\
\end{pmatrix}
\cdot
\begin{pmatrix}
1 & 0\\
0 & 1\\
-2 & -2\\
\end{pmatrix}
=
\begin{pmatrix}
1 & 0\\
0 & 1\\
0 & 0\\
\end{pmatrix},
\ \ 
\begin{pmatrix}
1 & 0 & 0\\
0 & 1 & 0\\
1 & 1 & 1\\
\end{pmatrix}
\cdot
\begin{pmatrix}
2 & 0\\
0 & 2\\
-2 & -2\\
\end{pmatrix}
=
\begin{pmatrix}
2 & 0\\
0 & 2\\
0 & 0\\
\end{pmatrix}
\ee
Hence $\coker(\beta^\ast_1)\cong\Z$ and $\coker(\beta^\ast_2)\cong\Z_2\oplus\Z_2\oplus\Z$. The maps $\beta^\vee_{1,2}$ are given by
\ben
\Bhalfarray{rcccccrcccc}
\beta^\vee_1:\ \Z^3 & \rightarrow & \coker(\beta^\ast_1) &\cong & \Z &\quad&
\beta^\vee_2:\ \Z^3 & \rightarrow & \coker(\beta^\ast_2) &\cong & \Z_2 \oplus \Z_2 \oplus \Z\\
e_1^\ast & \mapsto & e_1^\ast + 2 e_3^\ast & \leftrightarrow & 2 &\quad&
e_1^\ast & \mapsto & e_1^\ast + e_3^\ast & \leftrightarrow & (\bar1,\bar0,1) \\
e_2^\ast & \mapsto & e_2^\ast + 2 e_3^\ast & \leftrightarrow & 2 &\quad&
e_2^\ast & \mapsto & e_2^\ast + e_3^\ast & \leftrightarrow & (\bar0,\bar1,1) \\
e_3^\ast & \mapsto & e_3^\ast & \leftrightarrow & 1 &\quad&
e_3^\ast & \mapsto & e_3^\ast & \leftrightarrow & (\bar0,\bar0,1). \\
\Ehalfarray
\ee
So the groups $H_{1,2} = H(\beta_{1,2})$ act via $\TT$ in the following way:
\ben
\Bhalfarray{rclcrcl}
H_1 \cong \CC^\ast & \rightarrow & \TT &\quad&
H_2 \cong \ZZ_2 \times \ZZ_2 \times \CC^\ast& \rightarrow & \TT \\
t & \mapsto & (t^2,t^2,t) &\quad&
(\bar a,\bar b,t) & \mapsto & ((-1)^{\bar a} t, (-1)^{\bar b} t, t)\\
\Ehalfarray
\ee
Putting everything together we obtain from the stacky fan $\Stacky_{1,2}=(\Sigma,\Z^2,\beta_{1,2})$ the toric DM stack $\XX(\Stacky_{1,2}) = \left[ \quotient{ \C^3 \setminus \{0\}}{H_{1,2}} \right]$.
\subsubsection*{(ii) Weighted projective space}
For the fan of $\PP(1,1,2)$ (see Figure~\ref{fig:projweight}) we take $\beta$ trivial as in Section~\ref{sympstack}. Diagonalising the transpose of $M$ given in~\eqref{matrix2} yields
\ben
\begin{pmatrix}
1 & 0 & 0\\
0 & 1 & 0\\
1 & 2 & 1\\
\end{pmatrix}
\cdot
\begin{pmatrix}
1 & 0\\
0 & 1\\
-1 & -2\\
\end{pmatrix}
=
\begin{pmatrix}
1 & 0\\
0 & 1\\
0 & 0\\
\end{pmatrix}.
\ee
Hence $\coker(\beta^\ast)\cong\Z$ and the map $\beta^\vee$ is
\ben
\Bhalfarray{rcclcrccl}
\beta^\vee:\ \ZZ^3 & \rightarrow & \coker(\beta^\ast) &\cong\  \ZZ\\
e_1^\ast & \mapsto & e_1^\ast + e_3^\ast & \leftrightarrow\  1  \\
e_2^\ast & \mapsto & e_2^\ast + 2 e_3^\ast & \leftrightarrow\  2  \\
e_3^\ast & \mapsto & e_3^\ast & \leftrightarrow\  1. \\
\Ehalfarray
\ee
So $H$ acts via $\TT$ by
\ben
\Bhalfarray{rcl}
H=\CC^\ast & \rightarrow & \TT\\
t & \mapsto & (t,t^2,t)
\Ehalfarray
\ee
and we therefore obtain $\XX(\Stacky)=\left[ \quotient{\C^3\setminus\{0\}}{H}\right]$.
%
%
%
%
%
\section{Comparison of symplectic and complex toric DM stacks}
Consider the normal fan $\Sigma$ of a given polytope and a choice of ray generators $\beta:\Z^m\to N$, $e_j\mapsto n_ju_j$. Our aim is to show that the differentiable stacks induced by $(\Sigma,\beta)$ following the construction in Sections~\ref{sympstack} and~\ref{anastack} are isomorphic, that is, they have Morita equivalent atlases. To apply the simplified criterion of Morita equivalence as given in Section~\ref{groupstack}, we first establish the inclusions
\ben
\ker \bar \beta \hookrightarrow  H(\beta) \text{ and } 
\mu_\Sigma^{-1}(\xi)\hookrightarrow\C^m_\Sigma
\ee
(where $\mu_\Sigma$ is moment map induced by $\Sigma$, see below).

\begin{lem}
For $\beta:\Z^m \rightarrow N$ with finite cokernel, $\ker\bar\beta$ and
$\Hom(\coker\beta^\ast,\R/\Z)$ are naturally isomorphic.
\end{lem}

\begin{proof}
We apply the exact functor $\Hom(-,\R/\Z)$ to the sequence~\eqref{ex:theSeqStack}:
\ben
\xymatrix{
0 \ar@{=}[d] & N \otimes \quotient{\R}{\Z} \ar[l] \ar@{=}[d] & \left( \quotient{\R}{\Z} \right)^m \ar[l] \ar@{=}[d] 
& \Hom\left( \coker \beta^\ast, \quotient{\R}{\Z} \right) \ar[l] \ar@{=}[d] & 0 \ar[l]\ar@{=}[d] \\
0 & \quotient{N_\R}{N} \ar[l] & \quotient{\R^m}{\Z^m} \ar[l]^{\bar \beta} 
& \ker \bar \beta \ar[l]  & 0 \ar[l]
}
\ee
\end{proof}

Using the natural isomorphism $\C^\ast=S^1\times\R_+\cong\R/\Z\times\R$ provided by the exponential map we get the first inclusion
\begin{eqnarray*}
\ker \bar \beta & = & \Hom\left( \coker \beta^\ast, \quotient{\RR}{\ZZ} \right)\\
& \subset & \Hom\left( \coker \beta^\ast, \quotient{\RR}{\ZZ} \right) \times
\underbrace{\Hom\left( \coker \beta^\ast, \RR \right)}_{=: C_\RR} \cong
H(\beta)
\end{eqnarray*}

For the second inclusion, we apply as above $-\otimes\R$ to the sequence~\eqref{ex:theSeqStack} and get
\ben
0 \rightarrow M_\R\rightarrow(\R^m)^\ast\stackrel{\beta^\vee_\R}{\longrightarrow}
\coker\beta^\ast\otimes\R\rightarrow0.
\ee
We compose the moment map $\mu_0:\C^m\rightarrow(\R^m)^\ast,(z_1, \ldots, z_m)\mapsto
(|z_1|^2, \ldots, |z_m|^2)/2$ with the map $\beta^\vee_\R$ and obtain the moment map
\ben
\mu_\Sigma:\C^m\stackrel{\mu_0}{\longrightarrow}(\R^m)^\ast\stackrel{\beta^\vee_\R}{\longrightarrow}\coker\beta^\ast\otimes\R
\ee
used for the reduction.

\begin{lem}
The map $\mu_\Sigma$ does not depend on the ``stacky'' information, i.e.\ on the chosen
$n_j\in\N$ in the definition of $\beta$. In particular we obtain the same map for the trivial labelling $n_j=1$, $j=1,\ldots,m$.
\end{lem}

\begin{proof}
The map $\beta^\vee$ is defined as the canonical projection
$(\ZZ^m)^\ast \rightarrow \quotient{(\ZZ^m)^\ast}{\im \beta^\ast}$.
Let $\beta_0$ be the map defined by $e_i \mapsto u_i$.
After applying $- \otimes \RR$ the images $\im \beta^\ast_\RR$ and
$\im \beta^\ast_{0,\RR}$ are equal, hence $\beta^\vee_\RR$ and
$\beta^\vee_{0,\RR}$ are equal.
\end{proof}

The map $\beta_0$ together with the fan $\Sigma$ are just the data for the usual Cox construction of the toric variety $X(\Sigma)$. As the definition of $\C^m_\Sigma$ is also independent of $\beta$ we obtain the second inclusion $\mu_\Sigma^{-1}(\xi) \subset \CC^m_\Sigma$, since this is already known for toric varieties, see Theorem 1.4 in Appendix 1 of~\cite{gui}.

\medskip

The next step requires a closer look at the group action. We set
\ben
\xymatrix@R=1pt{
U = \mu_\Sigma^{-1}(\xi) \ar@{^(->}[r] & V = \C^m_\Sigma \\
\circlearrowleft & \circlearrowleft \\
G = \ker \bar \beta  \ar@{^(->}[r] & 
H = H(\beta).
}
\ee
As observed above we can write $H = G \times C_\R$ with $C_\R= \Hom( \coker \beta^\ast, \R)$. 
Let $\coker \beta^\ast = \Z^l \oplus T$ be an arbitrary splitting into the free and the torsion part. Since $\R$ is torsion--free, $C_\R\cong \Hom( \Z^l, \R) = \R^l$. Hence $C_\R$ and its action on $V$  do not depend on the coefficients $n_j$ in the definition of $\beta$. In other words, all the stacky information of $\beta$ is already contained in $G$.

\begin{thm}
Let $(\Sigma,\beta)$ be a stacky fan. Then the two stacks
\ben
\left[ \quotient{\mu_\Sigma^{-1}(\xi)}{\ker \bar \beta} \right]
\text{ and }
\left[ \quotient{\CC^m_\Sigma}{H(\beta)} \right]
\ee
are isomorphic.
\end{thm}

\begin{proof}
First we show that the inclusion of stacks 
$\left[ \quotient{U}{G} \right] \hookrightarrow \left[ \quotient{V}{H} \right]$ 
satisfies \refmoritadiagp:
\ben
\forall h \in H \text{ and } u \in U\ :\ [h \bullet u \in U \Rightarrow h \in G].
\ee
Since $H$ splits into $H = G \times C_\R$ we only need to test whether an $h\in C_\R$ with $h \bullet u \in U$ is necessarily zero. But this is independent of the specific coefficients $n_j$ in $\beta$, so again we deduce the result from the already known case of toric varieties.

\medskip

That the inclusion of stacks also satisfies~\refmoritasurj:
\ben
H \times U \stackrel{\bullet}{\longrightarrow} V \text{ is surjective,}
\ee
follows from $V=C_\R\times U$. This equality holds since all three ingredients
are independent of the stacky information.
\end{proof}

\begin{rem}
Since the maps commute with the natural torus actions we actually have an {\em equivariant} isomorphism between $[\mu_\Sigma^{-1}(\xi)/\ker\bar\beta]$ and $[\C^m_\Sigma/H(\beta)]$.
This equivariant isomorphism descends to a homeomorphism of the coarse moduli spaces $\mu_\Sigma^{-1}(\xi)/\ker\bar\beta$ and $\C^m_\Sigma/H(\beta)$.
\end{rem}

\begin{center}
{\bf Acknowledgments}
\end{center}
We would like to thank Lars Petersen, David Ploog and the referee for valuable comments on the manuscript.

\bigskip

{\em Added in proof.} After submission of the manuscript we became aware of Sakai's article~\cite{s} which also treats the correspondence between complex and symplectic toric DM stacks.

\bigskip

\noindent 
A.~{\sc Hochenegger}: Mathematisches Institut der Freien Universit\"at Berlin, Arnimallee 3, D--14195 Berlin, F.R.G.\\
e-mail: \texttt{hochen@math.fu-berlin.de}

\smallskip

F.~{\sc Witt}: Mathematisches Institut der Universit\"at M\"unster, Einsteinstra{\ss}e 62, D--48149 M\"unster, F.R.G.\\
e-mail: \texttt{frederik.witt@uni-muenster.de}

\begin{thebibliography}{99}
\bibitem[At]{at}
{\sc M.~Atiyah},
{\em Convexity and commuting Hamiltonians},
Bull. London Math. Soc. {\bf14} no.~1 (1982), 1--15. 

\bibitem[Au]{aud}
{\sc M.~Audin},
{\em The topology of torus actions on symplectic manifolds},
Progress in Mathematics {\bf 93}, Birkh\"auser, Boston, 1991.

\bibitem[BX]{bx}
{\sc K.~Behrend and P.~Xu},
{\em Differentiable Stacks and Gerbes},
J. Symplectic Geom. {\bf9} no.~3 (2011), 285--341.


\bibitem[BCS]{bcs} 
{\sc L.~Borisov, L.~Chen and G.~Smith},
{\em The orbifold Chow ring of toric Deligne-Mumford stacks}, 
J. Amer. Math. Soc. {\bf18} no.~1 (2005), 193--215.

\bibitem[CdS]{cds}
{\sc A.~Cannas da Silva},
{\em Lectures on symplectic geometry},
Lecture Notes in Mathematics {\bf1764}, Springer, Berlin, 2001. 

\bibitem[C]{cox} 
{\sc D.~Cox},
{\em The homogeneous coordinate ring of a toric variety},
J. Algebraic Geom. {\bf4} no.~1 (1995), 17--50.
 
\bibitem[CLS]{cls} 
{\sc D.~Cox, J.~Little and H.~Schenk},
{\em Toric Varieties},
to appear in the Graduate Studies in Mathematics series of the AMS. 

\bibitem[CM]{cm}
{\sc M.~Crainic and I.~Moerdijk},
{\em Foliation groupoids and their cyclic homology},
Adv. Math. {\bf157} no.~2 (2001), 177--197. 

\bibitem[D]{del}
{\sc T.~Delzant}, 
{\em Hamiltoniens p\'eriodiques et images convexes de l'application moment},
Bull. Soc. Math. France {\bf116} no.~3 (1988), 315--339.

\bibitem[F]{f}
{\sc B.~Fantechi},
{\em Stacks for everybody},
European Congress of Mathematics, Vol. I (Barcelona, 2000), 349--359, Progress in Mathematics {\bf201}, Birkh\"auser, Basel, 2001. 

\bibitem[FMN]{fmn}
{\sc B.~Fantechi, E.~Mann and F.~Nironi},
{\em Smooth Toric DM Stacks},
J. Reine Angew. Math. {\bf648} (2010), 201--244.

\bibitem[Gr]{g}
{\sc A.~Grothendieck},
{\em Rev\^etements \'etales et groupe fondamental}, 
S\'eminaire de g\'eom\'etrie alg\'ebrique du Bois-Marie 1960--1961 (SGA 1), Lecture Notes in Mathematics {\bf224}, Springer, Berlin, 1971.

\bibitem[Gu]{gui}
{\sc V.~Guillemin},
{\em Moment maps and combinatorial invariants of Hamiltonian $T^n$--spaces}, 
Progress in Mathematics {\bf 122}, Birkh\"auser, Boston, 1994. 

\bibitem[GS]{gs}
{\sc V.~Guillemin and S.~Sternberg},
{\em Convexity properties of the moment mapping},
Invent. Math. {\bf67} no.~3 (1982), 491--513.

\bibitem[H]{h}
{\sc J.~Heinloth},
{\em Notes on differentiable stacks}, 
Math. Inst. Georg--August--Universit\"at G\"ottingen: Seminars Winter Term 2004/2005, 1--32, Universit\"atsdrucke G\"ottingen, G\"ottingen, 2005. 

\bibitem[HM]{hm}
{\sc A.~Henriques and D.~Metzler},
{\em Presentations of noneffective orbifolds},
Trans. Am. Math. Soc. {\bf356}, no.~6 (2004), 2481--2499.
 
\bibitem[HS]{hs}
{\sc A.~Haefliger and \'E.~Salem},
{\em Actions of tori on orbifolds},
Ann. Global Anal. Geom. {\bf9} no.~1 (1991), 37--59.

\bibitem[I]{i}
{\sc I.~Iwanari},
{\em The category of toric stacks},
Compos. Math. {\bf145} no.~3 (2009), 718--746.

\bibitem[LM]{lm}
{\sc E.~Lerman and A.~Malkin},
{\em Hamiltonian group actions on symplectic Deligne-Mumford stacks and toric orbifolds},
Adv. Math. {\bf229} no.~3 (2012), 984--1000.
 
\bibitem[LT]{lt} 
{\sc E.~Lerman and S.~Tolman},
{\em Hamiltonian torus actions on symplectic orbifolds and toric varieties},
Trans. Amer. Math. Soc. {\bf349} no.~10 (1997), 4201--4230.

\bibitem[MM]{mm}
{\sc I.~Moerdijk and J.~Mr\v cun},
{\em Introduction to foliations and Lie groupoids},
Cambridge Studies in Advanced Mathematics, {\bf91}, CUP, Cambridge, 2003.

\bibitem[R]{r}
{\sc M.~Romagny}, 
{\em Group actions on stacks and applications},
Michigan Math. J. {\bf53} no.~1 (2005), 209--236.

\bibitem[S]{s}
{\sc H.~Sakai},
{\em The symplectic Deligne--Mumford stack associated to a stacky polytope},
to appear in: Results Math.
\end{thebibliography}
\end{document}